\documentclass[a4paper,11pt]{amsart}
\usepackage[active]{srcltx}

\usepackage{graphicx}
\usepackage{amsmath}
\usepackage{amssymb}
\usepackage{hyperref}

\newtheorem{thm}{Theorem}
\newtheorem{lem}[thm]{Lemma}%
\newtheorem{prop}[thm]{Proposition}%
\newtheorem{cor}[thm]{Corollary}%
\theoremstyle{remark}
\newtheorem{remark}{Remark}[section] %

\theoremstyle{plain}

\numberwithin{equation}{section}

\def\RR{{\mathbb R}}

\def\ZZ{{\mathbb Z}}

\def\Span{\operatorname{Span}}

\def\vecb{{\text{\boldmath$b$}}}

\def\vece{{\text{\boldmath$e$}}}

\def\vecell{{\text{\boldmath$\ell$}}}
\def\vecm{{\text{\boldmath$m$}}}
\def\vecn{{\text{\boldmath$n$}}}
\def\vecq{{\text{\boldmath$q$}}}

\def\vecp{{\text{\boldmath$p$}}}

\def\vecv{{\text{\boldmath$v$}}}

\def\vecw{{\text{\boldmath$w$}}}
\def\vecx{{\text{\boldmath$x$}}}

\def\vecy{{\text{\boldmath$y$}}}

\def\vecgamma{{\text{\boldmath$\gamma$}}}

\def\vecxi{{\text{\boldmath$\xi$}}}

\def\vecnull{{\text{\boldmath$0$}}}

\def\scrA{{\mathcal A}}
\def\scrB{{\mathcal B}}

\def\scrD{{\mathcal D}}
\def\scrE{{\mathcal E}}
\def\scrF{{\mathcal F}}

\def\scrL{{\mathcal L}}

\def\scrN{{\mathcal N}}
\def\scrO{{\mathcal O}}

\def\scrP{{\mathcal P}}

\def\scrV{{\mathcal V}}
\def\scrW{{\mathcal W}}

\def\fC{{\mathfrak C}}
\def\fD{{\mathfrak D}}

\def\fS{{\mathfrak S}}

\def\fU{{\mathfrak U}}

\def\e{\mathrm{e}}
\def\i{\mathrm{i}}

\def\diag{\operatorname{diag}}
\def\dim{\operatorname{dim}}

\def\intl{{\operatorname{int}}}

\def\L{\operatorname{L{}}}

\def\S{\operatorname{S{}}}

\def\SL{\operatorname{SL}}
\def\ASL{\operatorname{ASL}}

\def\SO{\operatorname{SO}}

\def\vol{\operatorname{vol}}

\def\GamG{\Gamma\backslash G}

\def\trans{\,^\mathrm{t}\!}

\def\hatP{{\widehat{\mathcal P}}}

\def\Onder#1#2#3#4#5{#1 \setbox0=\hbox{$#1$}\setbox1=\hbox{$#2$}
       \dimen0=.5\wd0 \dimen1=\dimen0 \dimen2=\dp0 \dimen3=\dimen2
       \advance\dimen0 by .5\wd1 \advance\dimen0 by -#4
       \advance\dimen1 by -.5\wd1 \advance\dimen1 by -#4
       \advance\dimen2 by -#3 \advance\dimen2 by \ht1
       \advance\dimen2 by 0.3ex \advance\dimen3 by #5
        \kern-\dimen0\raisebox{-\dimen2}[0ex][\dimen3]{\box1}
       \kern\dimen1}

\newcommand{\Q}{\mathbb{Q}}
\newcommand{\R}{\mathbb{R}}
\newcommand{\Z}{\mathbb{Z}}

\newcommand{\col}{\: : \:}

\newcommand{\bn}{\mathbf{0}}

\newcommand{\ve}{\varepsilon}

\newcommand{\matr}[4]{\left( \begin{matrix} #1 & #2 \\ #3 & #4 \end{matrix} \right) }
\newcommand{\smatr}[4]{\bigr( \begin{smallmatrix} #1 & #2 \\ #3 & #4 \end{smallmatrix} \bigr) }

\title{Visibility and directions in quasicrystals}
\author{Jens Marklof}
\author{Andreas Str\"ombergsson}
\address{School of Mathematics, University of Bristol,
Bristol BS8 1TW, U.K.\newline
\rule[0ex]{0ex}{0ex} \hspace{8pt}{\tt j.marklof@bristol.ac.uk}}
\address{Department of Mathematics, Box 480, Uppsala University,
SE-75106 Uppsala, Sweden\newline
\rule[0ex]{0ex}{0ex} \hspace{8pt}{\tt astrombe@math.uu.se}}
\date{5 April 2014/6 August 2014. Accepted for publication in IMRN}
\thanks{The research leading to these results has received funding from the European Research Council under the European Union's Seventh Framework Programme (FP/2007-2013) / ERC Grant Agreement n. 291147. 
J.M.\ is furthermore supported by a Royal Society Wolfson Research Merit Award, and
A.S.\ is supported by a grant from the G\"oran Gustafsson Foundation for Research in Natural Sciences and Medicine.}

\begin{document}

\begin{abstract}
It is well known that a positive proportion of all points in a $d$-dimensional lattice is visible from the origin, and that these visible lattice points have constant density in $\RR^d$. In the present paper we prove an analogous result for a large class of quasicrystals, including the vertex set of a Penrose tiling. We furthermore establish that the statistical properties of the directions of visible points are described by certain $\SL(d,\RR)$-invariant point processes. Our results imply in particular existence and continuity of the gap distribution for directions in certain two-dimensional cut-and-project sets. This answers some of the questions raised by Baake et al.\  in [arXiv:1402.2818]. 
\end{abstract}

\maketitle

\section{Introduction}

A point set $\scrP\subset\RR^d$ has constant density in $\RR^d$ if there exists $\theta(\scrP)<\infty$ such that, for any bounded $\scrD\subset\RR^d$ with boundary of Lebesgue measure zero,
\begin{equation}\label{asyden}
\lim_{T\to\infty} \frac{ \#(\scrP\cap T \scrD)}{T^d} = \theta(\scrP) \vol(\scrD) .
\end{equation}
We refer to $\theta(\scrP)$ as the density of $\scrP$. It is interesting to compare the density of $\scrP$ with the density of the subset of {\em visible} points given by
\begin{align}\label{VISPDEF}
\hatP=\bigl\{\vecy\in\scrP\col t\vecy\notin\scrP\:\forall t\in(0,1)\bigr\}.
\end{align}
This definition assumes that the observer is at the origin $\vecnull$. Note also that, by definition, $\vecnull\notin\hatP$.
A classic example is the set of integer lattice points $\scrP=\ZZ^d$. In this case, the set of visible points is given by the primitive lattice points $\hatP=\{\vecm\in\ZZ^d\col \gcd(\vecm)=1\}$. Both sets have constant density with $\theta(\scrP)=1$ and $\theta(\hatP)=1/\zeta(d)$, where $\zeta(d)$ denotes the Riemann zeta function. 

In this paper we are interested in the visible points of a regular cut-and-project set $\scrP=\scrP(\scrW,\scrL)$ constructed from a (possibly affine) lattice $\scrL\subset\RR^{d+m}$ and a window set $\scrW\subset\RR^m$ (see Section \ref{sec:two} for detailed definitions). Such $\scrP$ are also referred to as (Euclidean) model sets. Our first observation is the following.

\begin{thm}\label{thm:main1}
If $\scrP=\scrP(\scrW,\scrL)$ is a regular cut-and-project set, then $\scrP$ and $\hatP$ have constant density with $0<\theta(\hatP)\leq \theta(\scrP)$.
\end{thm}

The constant density of $\scrP$ is a well known fact \cite{Hof98,Schlottmann,Moody02}. The main point of Theorem \ref{thm:main1} is that the {\em visible} set $\hatP$ also has a strictly positive constant density. Although for cut-and-project sets $\scrP$ with generic choices of $\scrL$ we have $\theta(\hatP)=\theta(\scrP)$, there are important examples with $\theta(\hatP)< \theta(\scrP)$. The Penrose tilings and other cut-and-project sets which are based on the construction in \cite[Sec.\ 2.2]{qc} 
fall into this category, cf.~\cite{Baake13,Pleasants03}. In some special cases, such as the Ammann-Beenker model, the visible set $\hatP$ can be explicitly described by a simple condition in the cut-and-project construction, see \cite[Ch.\ 10.4]{Baake13} for details.

The second result of this paper concerns the distribution of directions in $\scrP$. Consider a general point set with constant density $\theta(\scrP)>0$ ($\scrP$ may be the visible set itself). 
We write $\scrP_T=\scrP\cap\scrB^d_T\setminus\{\bn\}$ for the subset of points lying in the punctured open ball of radius $T$, centered at the origin.
The number of such points is $\#\scrP_T\sim v_d \, \theta(\scrP)\, T^d$ as $T\to\infty$, where 
$v_d=\vol(\scrB^d_1)=\pi^{d/2}/\Gamma(\tfrac{d+2}{2})$ is the volume of the unit ball. 
For each $T$, we study the directions $\|\vecy\|^{-1}\vecy\in\S_1^{d-1}$ with $\vecy \in\scrP_T$, counted {\em with} multiplicity (if $\scrP=\hatP$ then the multiplicity is naturally one). The asymptotics \eqref{asyden} implies that, as $T\to\infty$, the directions become uniformly distributed on $\S_1^{d-1}$. That is, for any set $\fU\subset\S_1^{d-1}$ with boundary of measure zero (with respect to the volume element $\omega$ on $\S_1^{d-1}$) we have
\begin{equation}\label{udi}
\lim_{T\to\infty} \frac{\# \{\vecy\in\scrP_T\col\|\vecy\|^{-1}\vecy\in\fU\}}
{\#\scrP_T} 
= \frac{\omega(\fU)}{\omega(\S_1^{d-1})} . 
\end{equation}
Recall that $\omega(\S_1^{d-1})=d\,v_d$.

To understand the fine-scale distribution of the directions in $\scrP_T$, we consider the probability of finding $r$ directions in a small open disc $\fD_T(\sigma,\vecv)\subset\S_1^{d-1}$ with random center $\vecv\in\S_1^{d-1}$ and volume $\omega(\fD_T(\sigma,\vecv))=\frac{\sigma d}{\theta(\scrP) T^{d}}$ with $\sigma>0$ fixed.
Denote by
\begin{equation}\label{disc00}
	\scrN_{T}(\sigma,\vecv,\scrP)=\#\{\vecy\in\scrP_T \col \|\vecy\|^{-1}\vecy\in \fD_T(\sigma,\vecv)\}
\end{equation}
the number of points in $\fD_T(\sigma,\vecv)$.
The scaling of the disc size ensures that 
the expectation value for the counting function is asymptotically equal to 
$\sigma$. That is, for any probability measure $\lambda$ on $\S_1^{d-1}$ with continuous density,
\begin{equation}\label{expval}
\lim_{T\to\infty}	\int_{\S_1^{d-1}} \scrN_{T}(\sigma,\vecv,\scrP)\, d\lambda(\vecv) =\sigma .
\end{equation}
This fact follows directly from \eqref{asyden}. In the following, we denote by 
\begin{equation}\label{rel-dens}
\kappa_\scrP:=\frac{\theta(\hatP)}{\theta(\scrP)}
\end{equation}
the relative density of visible points in $\scrP$. We will prove:

\begin{thm}
\label{thm:main2}
Let $\scrP=\scrP(\scrW,\scrL)$ be a regular cut-and-project set, $\sigma>0$, $r\in\ZZ_{\geq 0}$, and let $\lambda$ be a Borel probability measure on $\S_1^{d-1}$ which is absolutely continuous with respect to $\omega$.
Then the limits
\begin{equation} \label{eq:main21}
	E(r,\sigma,\scrP):=\lim_{T\to\infty} \lambda(\{ \vecv\in\S_1^{d-1} : \scrN_{T}(\sigma,\vecv,\scrP)=r \}),
\end{equation}
\begin{equation} \label{eq:main22}
	E(r,\sigma,\hatP):=\lim_{T\to\infty} \lambda(\{ \vecv\in\S_1^{d-1} : \scrN_{T}(\sigma,\vecv,\hatP)=r \})
\end{equation}
exist, are continuous in $\sigma$ and independent of $\lambda$. For $\sigma\to 0$ we have
\begin{equation}\label{eq:main23}
E(0,\sigma,\scrP) = 1 - \kappa_\scrP \, \sigma +o(\sigma),
\end{equation}
\begin{equation} \label{eq:main24}
E(0,\sigma,\hatP) = 1 -  \sigma +o(\sigma).
\end{equation}
\end{thm}

This theorem generalizes our previous work on directions in Euclidean lattices \cite[Section 2]{partI}.
The existence of the limit \eqref{eq:main21} has already been established in \cite[Thm.\ A.1]{qc}. It is worthwhile noting that, if the set of directions in $\scrP$ were independent and uniformly distributed random variables in $\S_1^{d-1}$, then \eqref{eq:main21} would converge almost surely to the Poisson distribution
\begin{equation}
E(r,\sigma) = \frac{\sigma^r}{r!} \, \e^{-\sigma}.
\end{equation}
Although \eqref{eq:main24} is consistent with the Poisson distribution, we will see in Section \ref{sec:random} that $E(r,\sigma,\hatP)$ is characterized by a certain point process in $\RR^d$ which is determined by a finite-dimensional probability space.

Since $\lambda$ is arbitrary in Theorem \ref{thm:main2},
the result can readily be extended to cases where $\scrP$ is exhausted by more 
general expanding $d$-dimensional domains in place of the balls $\scrB_T^d$.
We make this precise in the appendix.

Theorem \ref{thm:main2} allows us to answer a recent question of Baake et al.~\cite{BGHJ} on the existence of the gap distribution for the directions in the class of two-dimensional cut-and-project sets considered here. In dimension $d=2$, it is convenient to identify the circle $\S_1^1$ with the unit interval mod 1, and represent the set of directions in $\scrP_T$ as $\tfrac1{2\pi}\arg\bigl(y_1+\i y_2)$ with $\vecy=(y_1,y_2)\in\scrP_T$.
We label these numbers (with multiplicity) in increasing order by
\begin{align}
-\tfrac12<\xi_{T,1}\leq\xi_{T,2}\leq\cdots\leq\xi_{T,N(T)}\leq\tfrac12,
\end{align}
where $N(T):=\#\scrP_T$. The analogous construction for the visible set $\hatP$ yields the multiplicity-free set of directions
\begin{align}\label{WHXI}
-\tfrac12<\widehat\xi_{T,1}<\widehat\xi_{T,2}<\cdots<\widehat\xi_{T,\widehat N(T)}\leq\tfrac12
\end{align}
where $\widehat N(T):=\#\hatP_T\leq N(T)$.
We also set $\xi_{T,0}=\widehat\xi_{T,0}=\xi_{T,N(T)}-1=\widehat\xi_{T,\widehat N(T)}-1$.
\begin{cor}\label{GAPDISTRCOR}
If $\scrP=\scrP(\scrW,\scrL)$ is a regular cut-and-project set
in dimension $d=2$, there exists a continuous decreasing function $F$ on $\R_{\geq0}$ 
satisfying $F(0)=1$ and $\lim_{s\to\infty}F(s)=0$,
such that for every $s\geq0$,
\begin{align}\label{GAPDISTRCORRES1}
\lim_{T\to\infty}\frac{\#\{1\leq j\leq \widehat N(T)\col \widehat N(T)(\widehat\xi_{T,j}-\widehat\xi_{T,j-1})\geq s\}}{\widehat N(T)}=F(s)
\end{align}
and
\begin{align}\label{GAPDISTRCORRES2}
\lim_{T\to\infty}\frac{\#\{1\leq j\leq N(T)\col N(T)(\xi_{T,j}-\xi_{T,j-1})\geq s\}}{N(T)}=
\begin{cases}1&\text{if }\: s=0
\\
\kappa_\scrP F(\kappa_\scrP s)&\text{if }\: s>0 .
\end{cases}
\end{align}
\end{cor}

It follows from the properties of $F(s)$ that 
the limit distribution function in \eqref{GAPDISTRCORRES2} is continuous at $s=0$ if and only if $\kappa_\scrP=1$.

In the special case when $\scrP=\Z^2$, \eqref{GAPDISTRCORRES1} was proved earlier by Boca, Cobeli and Zaharescu \cite{Boca00},
who also gave an explicit formula for the limit distribution.
More generally for $\scrP$ any affine lattice in $\R^2$, Corollary \ref{GAPDISTRCOR} was proved in 
\cite[Thm.\ 1.3, Cor.\ 2.7]{partI}.

Baake et al.~\cite{BGHJ} have observed numerically that the limiting gap distribution in Corollary~\ref{GAPDISTRCOR} may vanish near zero. In Section \ref{sec:twelwe} we will explain this hard-core repulsion between visible directions in the case of two-dimensional cut-and-project sets constructed over algebraic number fields, including any $\scrP$ associated with a Penrose tiling. There is, however, no hard-core repulsion for {\em typical} two-dimensional cut-and-project sets. The phenomenon can be completely ruled out in higher dimensions $d\geq 3$, where we show that $E(0,\sigma,\hatP)> 1-\sigma$ for all $\sigma>0$.

The organization of this paper is as follows. In Section \ref{sec:two} we recall the definition of a cut-and-project set of a higher-dimensional lattice. In Section \ref{sec:random} we construct random point processes in $\RR^d$ whose realizations yield the visible points in certain $\SL(d,\RR)$-invariant families of cut-and-project sets. These point processes describe the limit distributions in Theorem \ref{thm:main2}, cf.~Theorem \ref{visThm0} in Section \ref{sec:random}. This follows closely  the construction in \cite{qc} for the full cut-and-project set. An important technical tool in our approach is the Siegel-Veech formula, which is stated and proved in Section \ref{sec:SV}. In Section \ref{sec:five} we describe the small-$\sigma$ asymptotics of the void distribution in \eqref{eq:main23} and \eqref{eq:main24}. Sections \ref{sec:six}--\ref{sec:nine} are devoted to the proof of Theorem \ref{thm:main1}, Sections \ref{sec:ten} and \ref{sec:eleven} to the proofs of Theorem \ref{thm:main2} and Corollary \ref{GAPDISTRCOR}, respectively.
Finally in Section \ref{sec:twelwe} we discuss the possible vanishing of the limiting gap distribution near zero.

\section{Cut-and-project sets}\label{sec:two}

We start by recalling the definition of a cut-and-project set in $\R^d$ using our notation in \cite{qc}. These sets are also known as (Euclidean) model sets. We refer the reader to the recent monograph \cite{Baake13} and the surveys \cite{Moody97,Moody00} for a comprehensive introduction.

Denote by $\pi$ and $\pi_\intl$ the orthogonal projection of $\RR^n=\RR^d\times\RR^m$ 
onto the first $d$ and last $m$ coordinates.
We refer to $\RR^d$ and $\RR^m$ as the {\em physical space} and {\em internal space}, respectively. 
Let $\scrL\subset\RR^n$ be a lattice of full rank.
Then the closure of the set $\pi_\intl(\scrL)$ is an abelian subgroup $\scrA$ of $\RR^m$.
We denote by $\scrA^\circ$ the connected subgroup of $\scrA$ containing $\bn$;
then $\scrA^\circ$ is a linear subspace of $\R^m$,
say of dimension $m_1$, and there exist $\vecb_1,\ldots,\vecb_{m_2}\in\scrL$
($m=m_1+m_2$) such that $\pi_\intl(\vecb_1),\ldots,\pi_\intl(\vecb_{m_2})$ are linearly independent in
$\R^m/\scrA^\circ$ and 
\begin{align}
\scrA=\scrA^\circ+\Z\pi_\intl(\vecb_1)+\ldots+\Z\pi_\intl(\vecb_{m_2}).
\end{align}

Given $\scrL$ and a bounded subset $\scrW\subset\scrA$ with non-empty interior, we define
\begin{equation}\label{CUTPROJDEF}
	\scrP(\scrW,\scrL) = \{ \pi(\vecy) : \vecy\in\scrL, \; \pi_\intl(\vecy)\in\scrW \} \subset \RR^d .
\end{equation}
We will call $\scrP=\scrP(\scrW,\scrL)$ a {\em cut-and-project set}, and $\scrW$ the {\em window}. 
We denote by $\mu_\scrA$ the Haar measure of $\scrA$, normalized so that its restriction to $\scrA^\circ$
is the standard $m_1$-dimensional Lebesgue measure.
If $\scrW$ has boundary of measure zero with respect to $\mu_\scrA$, we will say $\scrP(\scrW,\scrL)$ is {\em regular}. 
Set $\scrV=\R^d\times\scrA^\circ$; 
then $\scrL_\scrV=\scrL\cap\scrV$ is a lattice of full rank in $\scrV$.
Let $\mu_\scrV=\vol\times\mu_\scrA$ be the natural volume measure on $\R^d\times\scrA$
(this restricts to the standard $d+m_1$ dimensional Lebesgue measure on $\scrV$).
It follows from Weyl equidistribution (see \cite{Hof98} or \cite[Prop.\ 3.2]{qc})
that for any regular cut-and-project set $\scrP$ and 
any bounded $\scrD\subset\RR^d$ with boundary of measure zero with respect to Lebesgue measure, 
\begin{equation}\label{WEYLCOUNTING1}
\lim_{T\to\infty} \frac{\#\{ \vecb \in \scrL\col \pi(\vecb)\in\scrP\cap T \scrD\}}{T^d} 
= C_\scrP\vol(\scrD) 
\end{equation}
where
\begin{align}\label{DELTADMLDEF}
C_\scrP:=\frac{\mu_\scrA(\scrW)}{\mu_\scrV(\scrV/\scrL_\scrV)}.
\end{align}
A further condition often imposed in the quasicrystal literature is that $\pi|_\scrL$ is injective (i.e., the map $\scrL\to \pi(\scrL)$ is one-to-one); we will not require this here. To avoid coincidences in $\scrP$, we assume throughout this paper that the window is appropriately chosen so that the map $\pi_\scrW: \{ \vecy\in\scrL : \pi_\intl(\vecy)\in\scrW \}\to \scrP$ is bijective. Then \eqref{WEYLCOUNTING1} implies
\begin{equation}\label{density000}
\lim_{T\to\infty} \frac{ \#(\scrP\cap T \scrD)}{T^d} = C_\scrP\vol(\scrD),
\end{equation}
i.e., $\scrP$ has density $\theta(\scrP)=C_\scrP$.
Under the above assumptions $\scrP(\scrW,\scrL)$ is a Delone set, i.e., uniformly discrete and relatively dense in $\RR^d$.

We furthermore extend the definition of cut-and-project sets $\scrP(\scrW,\scrL)$ to affine lattices $\scrL=\scrL_0+\vecx$ with $\vecx\in\RR^n$ and $\scrL_0$ a lattice; note that $\scrP(\scrW,\scrL+\vecx)=\scrP(\scrW-\pi_\intl(\vecx),\scrL)+\pi(\vecx)$.

\section{Random cut-and-project sets}\label{sec:random}

Following our approach in \cite{qc}, we will now, for any given regular cut-and-project set $\scrP=\scrP(\scrW,\scrL)$, construct two
$\SL(d,\R)$-invariant random point processes on $\R^d$ which will describe the limit distributions in Theorem \ref{thm:main2}.
Let $G=\ASL(n,\R)=\SL(n,\R)\ltimes\R^n$, with multiplication law
\begin{align}
(M,\vecxi)(M',\vecxi')=(MM',\vecxi M'+\vecxi').
\end{align}
Also set $\Gamma=\ASL(n,\Z)\subset G$.
Choose $g\in G$ and $\delta>0$ so that $\scrL=\delta^{1/n}(\Z^ng)$,
and let $\varphi_g$ be the embedding of $\ASL(d,\R)$ in $G$ given by
\begin{equation}
\varphi_g: \ASL(d,\RR) \to G,\quad (A,\vecx) \mapsto g \left( \begin{pmatrix} A  &  0 \\ 0 & 1_m \end{pmatrix},(\vecx,\vecnull) \right) g^{-1} .
\end{equation}
It then follows from Ratner's work \cite{Ratner91a,Ratner91b}
that there exists a unique closed connected subgroup $H_g$ of $G$ such that $\Gamma\cap H_g$ is a lattice in $H_g$, $\varphi_g(\SL(d,\R))\subset H_g$,
and the closure of $\Gamma\backslash\Gamma\varphi_g(\SL(d,\RR))$ in $\GamG$ is given by 
\begin{align}
X=\Gamma\backslash\Gamma H_g.
\end{align}
Note that $X$ can be naturally identified with the homogeneous space $(\Gamma\cap H_g)\backslash H_g$. 
We denote the unique right-$H_g$ invariant probability measure on either of these spaces by ${\mu}$;
sometimes we will also let ${\mu}$ denote the corresponding Haar measure on $H_g$.
For each $x=\Gamma h\in X$ we set
\begin{align}\label{Px:def}
\scrP^x:=\scrP(\scrW,\delta^{1/n}(\ZZ^n h g))
\end{align}
and denote by $\hatP^x$ the corresponding set of visible points.
Both sets are well defined since $\overline{\pi_\intl(\delta^{1/n}(\ZZ^n h g))}\subset\scrA$ for all $h\in H_g$;
in fact $\overline{\pi_\intl(\delta^{1/n}(\ZZ^n h g))}=\scrA$ for ${\mu}$-almost all $h\in H_g$;
cf.\ \cite[Prop.\ 3.5]{qc}. Note that $\scrP^x$ and $\hatP^x$ with $x$ random in $(X,{\mu})$ define random point processes on $\R^d$.
The fact that $\varphi_g(\SL(d,\R))\subset H_g$ implies that these processes are $\SL(d,\R)$-invariant. 

\begin{thm}\label{visThm0}
The limit distributions in Theorem \ref{thm:main2} are given by
\begin{equation}\label{eq01}
E(r,\sigma,\scrP)={\mu}(\{ x\in X\col \# (\scrP^x \cap \fC(\sigma))= r \}) 
\end{equation}
and
\begin{equation}\label{eq02}
E(r,\sigma,\hatP)={\mu}(\{ x\in X\col \# (\hatP^x \cap \fC(\kappa_\scrP^{-1}\sigma))= r \}) 
\end{equation}
where
\begin{equation} \label{FCCSDEF}
	\fC(\sigma) =\bigg\{(x_1,\ldots,x_d)\in\RR^d \col 0 < x_1 < 1, \: \|(x_2,\ldots,x_d)\|<
\Bigl(\frac{\sigma d}{C_\scrP v_{d-1}} 
\Bigr)^{1/(d-1)}x_1\biggr\}.
\end{equation}
\end{thm}

We note that relation \eqref{eq01} is a special case of \cite[Thm.\ A.1]{qc}. The new result of the present study is \eqref{eq02}.

In \cite[Section 1.4]{qc} we also consider the closed connected subgroup $\widetilde H_g$ of $G$ such that $\Gamma\cap \widetilde H_g$ is a lattice in $\widetilde H_g$, $\varphi_g(\ASL(d,\R))\subset\widetilde H_g$, and the closure of $\Gamma\backslash\Gamma\varphi_g(\ASL(d,\RR))$ in $\GamG$ is given by $\widetilde X:=\Gamma\backslash\Gamma\widetilde H_g$. The unique right-$\widetilde H_g$ invariant probability measure on $\widetilde X$ is denoted by $\widetilde\mu$. The point process $\scrP^x$ in \eqref{Px:def} with $x$ random in $(\widetilde X,{\widetilde \mu})$ is now $\ASL(d,\R)$-invariant, i.e., in addition to the previous $\SL(d,\R)$-invariance we also have translation-invariance. The latter implies that $\scrP^x=\hatP^x$ for $\widetilde\mu$-almost every $x\in\widetilde X$. Proposition 4.5 in \cite{qc} shows that for Lebesgue-almost all $\vecy\in\R^d\times\{\bn\}$ we have $H_{g(1_n,\vecy)}=\widetilde H_g$. This has the following interesting consequence.

\begin{cor}
Given any regular cut-and-project set $\scrP$ there is a subset $\fS\subset\R^d$ of Lebesgue measure zero such that
for every $\vecy\in\R^d\setminus\fS$
\begin{equation}\label{eq010}
E(r,\sigma,\scrP+\vecy)=E(r,\sigma,\widehat{\scrP+\vecy})={\widetilde\mu}(\{ x\in \widetilde X\col \# (\scrP^x \cap \fC(\sigma))= r \}) .
\end{equation}
\end{cor}

That is, all limit distributions are independent of $\vecy$ for Lebesgue-almost every $\vecy$.

\section{The Siegel-Veech formula for visible points}\label{sec:SV}

Throughout the remaining sections, we let $\scrP=\scrP(\scrW,\scrL)$ be a given regular cut-and-project set.
We fix $g\in G$ and $\delta>0$ so that $\scrL=\delta^{1/n}(\Z^ng)$.
In fact, by an appropriate scaling of the length units, we can assume without loss of generality that $\delta=1$.
This assumption will be in force throughout the remaining sections except the last one. 
Hence we now have $\scrP=\scrP(\scrW,\Z^ng)$ and $\scrP^x=\scrP(\scrW,\Z^nhg)$ for each $x=\Gamma h\in X$.

The following Siegel-Veech formulas will serve as a crucial technical tool in our proofs of the main theorems.

\begin{thm}\label{VEECHTHM}
For any $f\in\L^1(\R^d)$,
\begin{align}\label{VEECHTHMeq1}
\int_X\sum_{\vecq\in\scrP^x}f(\vecq)\,d{\mu}(x)=C_\scrP \int_{\R^d}f(\vecx)\,d\vecx
\end{align}
and
\begin{align}\label{VEECHTHMeq2}
\int_X\sum_{\vecq\in\hatP^x}f(\vecq)\,d{\mu}(x)=\kappa_\scrP\, C_\scrP \int_{\R^d}f(\vecx)\,d\vecx.
\end{align}
\end{thm}

Veech has proved formulas of the above type for general $\SL(d,\RR)$-invariant measures \cite[Thm.\ 0.12]{veech}. The proof of Theorem \ref{VEECHTHM} is simpler in the present setting.
Relation \eqref{VEECHTHMeq1} was proved in \cite[Theorem 1.5]{qc}.
In the present section we will prove that 
{\em there exists $0<\kappa_\scrP\leq1$ such that relation \eqref{VEECHTHMeq2} holds for all $f\in\L^1(\R^d)$}. We will then later establish that this $\kappa_\scrP$ indeed yields the relative density defined in \eqref{rel-dens}.

Consider the map
\begin{align}
B\mapsto\int_X\#(\hatP^x\cap B)\,d{\mu}(x)
\qquad (\text{$B$ any Borel subset of $\R^d$}).
\end{align}
This map defines a Borel measure on $\R^d$,
which is finite on any compact set $B$ (by \cite[Theorem 1.5]{qc}), 
invariant under $\SL(d,\R)$, 
and gives zero point mass to $\bn\in\R^d$. 
Hence up to a constant, the measure must equal Lebesgue measure,
i.e.\ there exists a constant $\kappa_\scrP\geq0$ such that
\begin{align}\label{KAPPADEFLEMRES}
\int_X\#(\hatP^x\cap B)\,d{\mu}(x)=\kappa_\scrP C_\scrP\vol(B)
\end{align}
for every Borel set $B\subset\R^d$.
By a standard approximation argument, this implies that \eqref{VEECHTHMeq2} holds for all $f\in\L^1(\R^d)$.
Also $\kappa_\scrP\leq1$ is immediate from \eqref{VEECHTHMeq1}.

It remains to verify that $\kappa_\scrP>0$.
Recall that we are assuming that $\scrW$ has non-empty interior $\scrW^\circ$ in $\scrA=\overline{\pi_\intl(\scrL)}$.
Now take $B$ to be any bounded open set in $\R^d$ which is star-shaped with center $\bn$ and such that
$(B\setminus\{\bn\})\times\scrW^\circ$ contains some point in the (affine) lattice $\scrL$.
Then the set of $x=\Gamma h$ in $X$ for which 
$\Z^n h g$ has at least one point in $(B\setminus\{\bn\})\times\scrW^\circ$
is non-empty and open.
Note that for any such $x$,
$\scrP^x=\scrP(\scrW,\Z^n h g)$ has a point in $B\setminus\{\bn\}$,
and hence also a \textit{visible} point in $B\setminus\{\bn\}$, since $B$ is star-shaped.
It follows that the left hand side of \eqref{KAPPADEFLEMRES} is positive for our set $B$.
Therefore $\kappa_\scrP>0$, as claimed.

\section{The limit distribution for small $\sigma$}\label{sec:five}

From now on we take $E(r,\sigma,\scrP)$ and $E(r,\sigma,\hatP)$ to be defined by the relations
\eqref{eq01}, \eqref{eq02}.
Then \eqref{eq:main21} holds by \cite[Thm.\ A.1]{qc},
and we will prove in Section \ref{sec:ten} that also \eqref{eq:main22} holds.

In the present section we will prove that the relation \eqref{eq:main23},
\begin{equation}\label{eq:main233}
E(0,\sigma,\scrP) = 1 - \kappa_\scrP \, \sigma +o(\sigma) ,
\end{equation}
holds with the same $\kappa_\scrP\in(0,1]$ as in the Siegel-Veech formula \eqref{VEECHTHMeq2}. Rel.~\eqref{eq:main24} is then a simple conseqence of the observation that
\begin{equation}\label{E0BASICREL}
E(0,\sigma,\hatP) = E(0,\kappa_\scrP^{-1} \sigma,\scrP) .
\end{equation}

To prove \eqref{eq:main233}, first note that, for any $\sigma>0$,
\begin{align}\notag
1-E(0,\sigma,\scrP)
=\mu\bigl(\bigl\{x\in X\col \scrP^x\cap\fC(\sigma)\neq\emptyset\bigr\}\bigr)
=\mu\bigl(\bigl\{x\in X\col \hatP^x\cap\fC(\sigma)\neq\emptyset\bigr\}\bigr)
\\\label{LEM1pf3}
\leq\int_X \#\bigl(\hatP^x\cap\fC(\sigma)\bigr)\,d{\mu}(x)
=\kappa_\scrP C_\scrP\vol(\fC(\sigma))=\kappa_\scrP \sigma,
\end{align}
where 
the integral was evaluated using \eqref{KAPPADEFLEMRES}.

On the other hand using the fact that the point process $\scrP^x$ ($x\in(X,\mu)$) is invariant under $\SO(d)$,
and $\widehat{\scrP' k}=\hatP' k$ for every point set $\scrP'$ and every $k\in\SO(d)$, we have
\begin{align}
1-E(0,\sigma,\scrP)=\int_X A(\sigma,\scrP^x)\,d{\mu}(x)
\end{align}
with
\begin{align}
A(\sigma,\scrP^x)=\int_{\SO(d)}I\Bigl(\hatP^x\cap\fC(\sigma)k\neq\emptyset\Bigr)\,dk,
\end{align}
where $dk$ is Haar measure on $\SO(d)$ normalized by $\int_{\SO(d)}\,dk=1$.

We write $\varphi(\vecp,\vecq)\in[0,\pi]$ for the angle between any two points $\vecp,\vecq\in\R^d\setminus\{\bn\}$,
as seen from $\bn$.
Also for any $x\in X$ 
we set
\begin{align}
\sigma_0(\scrP^x)=\frac{C_\scrP v_{d-1}}d \Bigl(\tan\frac{\varphi_0(\scrP^x)}2\Bigr)^{d-1}
\end{align}
where
\begin{align}
\varphi_0(\scrP^x)=\min\bigl\{\varphi(\vecp,\vecq)\col \vecp,\vecq\in\hatP^x\cap\scrB_1^d,\: \vecp\neq\vecq\bigr\},
\end{align}
with the convention that $\varphi_0(\scrP^x)=\pi$ and $\sigma_0(\scrP^x)=+\infty$ whenever
$\#(\hatP^x\cap\scrB_1^d)\leq1$.
These are measurable functions on $X$, and
$\varphi_0(\scrP^x)>0$ and $\sigma_0(\scrP^x)>0$ for all $x\in X$.

Now if $0<\sigma<\sigma_0(\scrP^x)$ then for any two distinct points $\vecp,\vecq\in\hatP^x\cap\scrB_1^d$
we have
\begin{align}
\varphi(\vecp,\vecq)>2\arctan\biggl(\Bigl(\frac{\sigma d}{C_\scrP v_{d-1}}\Bigr)^{1/(d-1)}\biggr),
\end{align}
and because of the definition of $\fC(\sigma)$, \eqref{FCCSDEF},
this implies that there does not exist any $k\in\SO(d)$ for which
$\fC(\sigma)k$ contains both $\vecp$ and $\vecq$.
Hence for $0<\sigma<\sigma_0(\scrP^x)$ we have
(writing $\vece_1=(1,0,\ldots,0)\in\R^d$)
\begin{align}\notag
A(\sigma,\scrP^x)\geq\sum_{\vecp\in\hatP^x\cap\scrB_1^d}\int_{\SO(d)} I\Bigl(\vecp\in\fC(\sigma)k\Bigr)\,dk
=\#\bigl(\hatP^x\cap\scrB_1^d\bigr)\cdot\int_{\SO(d)}I\Bigl(\vece_1\in\fC(\sigma)k\Bigr)\,dk
\\\label{ASIGMAPXBELOW}
=\frac{\vol(\fC(\sigma)\cap\scrB_1^d)}{\vol(\scrB_1^d)}\#\bigl(\hatP^x\cap\scrB_1^d\bigr),
\end{align}
and here
\begin{align}
\frac{\vol(\fC(\sigma)\cap\scrB_1^d)}{\vol(\scrB_1^d)}\sim
\frac{\vol(\fC(\sigma))}{\vol(\scrB_1^d)}=\frac{\sigma}{v_dC_\scrP}\qquad\text{as }\:\sigma\to0.
\end{align}
Hence given any number $K<(v_d C_\scrP)^{-1}$, there is some $\sigma(K)>0$ such that for all $0<\sigma<\sigma(K)$ we have
\begin{align}
1-E(0,\sigma,\scrP)=\int_X A(\sigma,\scrP^x)\,d{\mu}(x)
\geq K\sigma\int_X I\bigl(\sigma<\sigma_0(\scrP^x)\bigr)
\#\bigl(\hatP^x\cap\scrB_1^d\bigr)\,d{\mu}(x).
\end{align}
Furthermore, by the Monotone Convergence Theorem and \eqref{KAPPADEFLEMRES},
\begin{align}
\lim_{\sigma\to0}\int_X I\bigl(\sigma<\sigma_0(\scrP^x)\bigr)
\#\bigl(\hatP^x\cap\scrB_1^d\bigr)\,d{\mu}(x)
=\int_X\#\bigl(\hatP^x\cap\scrB_1^d\bigr)\,d{\mu}(\scrP^x)
=\kappa_\scrP C_\scrP v_d.
\end{align}
We thus conclude
\begin{align}\label{LEM1pf4}
\liminf_{\sigma\to0}\frac{1-E(0,\sigma,\scrP)}{\sigma}\geq K\kappa_\scrP C_\scrP v_d.
\end{align}
The claim \eqref{eq:main233} follows from \eqref{LEM1pf3} and the fact that \eqref{LEM1pf4} holds for every $K<(v_d C_\scrP)^{-1}$.

\section{Lower bound on the density of visible points}\label{sec:six}

Combining \eqref{eq:main233} and \eqref{eq:main21}
(recall that the latter 
was proved in \cite[Thm.\ A.1]{qc}), 
we get the following lower bound on the density $\theta(\hatP)=\kappa_\scrP  C_\scrP$ in Theorem \ref{thm:main1}:
\begin{lem}\label{LEM2}
Let $\fU$ be any subset of $\S_1^{d-1}$ with boundary of measure zero (w.r.t.\ $\omega$),
and let $\scrD=\{\vecv\in\R^d\col 0<\|\vecv\|<1,\:\|\vecv\|^{-1}\vecv\in\fU\}$ be the corresponding sector in $\scrB_1^d$.
Then
\begin{align}
\liminf_{T\to\infty}\frac{\#(\hatP\cap T\scrD)}{T^d}\geq\kappa_\scrP  C_\scrP \vol(\scrD).
\end{align}
\end{lem}
\begin{proof}
We may assume $\omega(\fU)>0$, since otherwise $\vol(\scrD)=0$ and the lemma is trivial.
Let $\ve>0$ be given, and let $\fU_\ve^-\subset\S_1^{d-1}$ be the ``$\ve$-thinning'' of $\fU$,
that is
\begin{align}
\fU_\ve^-=\bigl\{\vecv\in\S_1^{d-1}\col \bigl[\varphi(\vecw,\vecv)<\ve\Rightarrow\vecw\in\fU\bigr],\:\forall\vecw\in\S_1^{d-1}
\bigr\}.
\end{align}
(Recall that $\varphi(\vecw,\vecv)\in[0,\pi]$ is the angle between $\vecw$ and $\vecv$ as seen from $\bn$.)
Then $\omega(\fU_\ve^-)\to\omega(\fU)$ as $\ve\to0$,
since $\fU$ by assumption is a Jordan measurable subset of $\S_1^{d-1}$.
From now on we assume that $\ve$ is so small that $\omega(\fU_\ve^-)>0$.
We let $\lambda$ be 
$\omega$ restricted to $\fU_\ve^-$ and normalized to be a probability measure;
thus $\lambda(B)=\omega(\fU_\ve^-)^{-1}\omega(B\cap\fU_\ve^-)$ for any Borel subset $B\subset\S_1^{d-1}$.

Now note that, by the definitions of $\scrN_{T}(\sigma,\vecv,\scrP)$ and $\hatP$,
for any $\sigma>0$, $T>0$ and $\vecv\in\S_1^{d-1}$
we have $\scrN_{T}(\sigma,\vecv,\scrP)>0$ if and only if there is some $\vecy\in\hatP\cap\scrB_T^d$ such that
$\|\vecy\|^{-1}\vecy\in\fD_T(\sigma,\vecv)$.
Furthermore, if $T$ is larger than a certain constant depending on $\sigma,\scrP,\ve$, then
$\fD_T(\sigma,\vecv)\subset\fU$ for every $\vecv\in\fU_\ve^-$,
meaning that $\|\vecy\|^{-1}\vecy\in\fD_T(\sigma,\vecv)$ implies $\vecy\in\R_{>0}\scrD$.
Hence for such $T$ and $\sigma$ we have
\begin{align}\notag
\lambda(\{\vecv\in\S_1^{d-1}\col\scrN_{T}(\sigma,\vecv,\scrP)>0\})
=\lambda\bigl(\bigl\{\vecv\in\S_1^{d-1}\col \bigl[\exists\vecy\in\hatP\cap\scrB_T^d
\col\|\vecy\|^{-1}\vecy\in\fD_T(\sigma,\vecv)\bigr]\bigr\}\bigr)
\hspace{40pt}
\\\notag
\leq\sum_{\vecy\in\hatP\cap T\scrD}
\lambda\bigl(\bigl\{\vecv\in\S_1^{d-1}\col \|\vecy\|^{-1}\vecy\in\fD_T(\sigma,\vecv)\bigr\}\bigr)
\leq\frac{\omega(\fD_T(\sigma,\vece_1))}{\omega(\fU_\ve^-)}\cdot\#\bigl(\hatP\cap T\scrD\bigr)
\hspace{20pt}
\\
=\frac{\sigma d}{\omega(\fU_\ve^-)C_\scrP T^d}\cdot\#\bigl(\hatP\cap T\scrD\bigr).
\end{align}
Hence, letting $T\to\infty$ and applying \eqref{eq:main21} we have, for any fixed 
$\sigma>0$,
\begin{align}
\liminf_{T\to\infty}\frac{\#\hatP\cap T\scrD}{T^d}\geq
\frac{\omega(\fU_\ve^-) C_\scrP}{d}\cdot\frac{1-E(0,\sigma,\scrP)}{\sigma}.
\end{align}
Letting $\sigma\to0$ in the right hand side and using \eqref{eq:main233}, this gives
\begin{align}
\liminf_{T\to\infty}\frac{\#\hatP\cap T\scrD}{T^d}\geq
\kappa_\scrP C_\scrP\frac{\omega(\fU_\ve^-)}{d}.
\end{align}
Finally letting $\ve\to0$ and using $\omega(\fU)/d=\vol(\scrD)$ we obtain the statement of the lemma.
\end{proof}


\section{Continuity in the space of cut-and-project sets}\label{sec:seven}

Next, in Lemma \ref{KEYCONTLEM} and Lemma \ref{KEYCONTLEM2}, we will prove that for almost all $x\in X$,
both $\scrP^x$ and $\hatP^x$ 
\textit{vary continuously} as we perturb $x$.

\begin{lem}\label{ALLORALMOSTNONELEM}
For any $\vecm\in\R^n$, if $\pi(\vecm hg)\neq\bn$ for some $h\in H_g$ then $\pi(\vecm hg)\neq\bn$
for ${\mu}$-almost all $h\in H_g$.
Similarly, for any $\vecm,\vecn\in\R^n$, if
$\dim\Span\{\pi(\vecn hg),\pi(\vecm hg)\}=2$ for some $h\in H_g$
then $\dim\Span\{\pi(\vecn hg),\pi(\vecm hg)\}=2$ for ${\mu}$-almost all $h\in H_g$.
\end{lem}
\begin{proof}
$H_g$ is a connected, real-analytic manifold; hence any real-analytic function on $H_g$ which 
does not vanish identically is non-zero almost everywhere.
The first part of the lemma follows by applying this principle to the coordinate functions
$h\mapsto \pi(\vecm hg)\cdot\vece_j$ for $j=1,\ldots,d$.
The second part of the lemma follows by applying the same principle to the functions
\begin{align}
h\mapsto (\pi(\vecm hg)\cdot\vece_i)(\pi(\vecn hg)\cdot\vece_j)-(\pi(\vecm hg)\cdot\vece_j)(\pi(\vecn hg)\cdot\vece_i),
\end{align}
for $1\leq i<j\leq d$.
\end{proof}

\begin{lem}\label{KEYCONTLEM}
For ${\mu}$-almost every $x\in X$, and for every bounded open set $U\subset\R^d$ with $\scrP^x\cap\partial U=\emptyset$,
there is an open set $\Omega\subset X$ with $x\in\Omega$ 
such that $\#(\scrP^{x'}\cap U)=\#(\scrP^x\cap U)$ for all $x'\in\Omega$.
\end{lem}

\begin{proof}
For each $\vecm\in\Z^n$, by an argument as in Lemma \ref{ALLORALMOSTNONELEM} we either have
$\vecm hg\neq\bn$ for almost all $h\in H_g$ or else $\vecm hg=\bn$ for all $h\in H_g$.
By taking $h=1$ we see that the latter property can hold for at most one $\vecm\in\Z^n$,
and if it holds then we necessarily have $\vecm=\bn g^{-1}$, 
and $H_g\subset g\SL(n,\R)g^{-1}$. 
If such an exceptional $\vecm$ exists we call it $\vecm_E$, and we set 
$(\Z^n)':=\Z^n\setminus\{\vecm_E\}$; otherwise we set $(\Z^n)':=\Z^n$.

Now consider the following two subsets of $H_g$:
\begin{align}
&S_1=\bigl\{h\in H_g\col
(\Z^n)'hg\cap(\R^d\times\partial\scrW)\neq\emptyset\bigr\};
\\
&S_2=\bigl\{h\in H_g\col 
\exists 
\vecell_1\neq\vecell_2\in\Z^nhg\cap\pi_\intl^{-1}(\scrW)
\:\text{ satisfying }\:\pi(\vecell_1)=\pi(\vecell_2)\bigr\}.
\end{align}
We have ${\mu}(S_1)=0$, by \cite[Theorem 5.1]{qc}.
Also ${\mu}(S_2)=0$, by 
\cite[Prop.\ 3.7]{qc} applied to $\scrW^\circ$.
We will prove the lemma by showing that 
for every $h\in H_g\setminus(S_1\cup S_2)$,
the point $x=\Gamma h\in X$ has the property described in the lemma.

Thus let $h\in H_g\setminus(S_1\cup S_2)$ be given, set $x=\Gamma h\in X$, and 
let $U$ be an arbitrary bounded open subset of $\R^d$ with boundary disjoint from $\scrP^x=\scrP(\scrW,\Z^nhg)$.
Assume that the desired property does \textit{not} hold.
Then there is a sequence $h_1,h_2,\ldots$ in $H_g$ tending to $h$ such that 
\begin{align}\label{KEYCONTLEMPF1}
\#(\scrP(\scrW,\Z^nh_jg)\cap U)\neq\#(\scrP(\scrW,\Z^nhg)\cap U),\qquad\forall j.
\end{align}
Let $F$ be the (finite) set 
\begin{align}\label{KEYCONTLEMPF2}
F=\bigl\{\vecm\in\Z^n\col\vecm hg\in U\times\scrW\bigr\}.
\end{align}
Note that $\vecm hg\in U\times\scrW^\circ$ for every $\vecm\in F\cap(\Z^n)'$, since $h\notin S_1$.
But $U\times\scrW^\circ$ is open; hence by continuity we also have
$\vecm h'g\in U\times\scrW^\circ$ for every $h'\in H_g$ sufficiently near $h$ and all $\vecm\in F\cap(\Z^n)'$.
Note also that if the exceptional point $\vecm_E$ exists and belongs to $F$ then
$\bn=\vecm_E h'g\in U\times\scrW$ for \textit{all} $h'\in H_g$.
Hence, for every $h'\in H_g$ near $h$ we have
\begin{align}
\scrP(\scrW,\Z^nh'g)\supset \{\pi(\vecm h'g)\col \vecm\in F\}.
\end{align}

Because of $h\notin S_2$, the points $\pi(\vecm hg)$ for $\vecm\in F$ are pairwise distinct.
By continuity it then also follows that for any $h'\in H_g$ sufficiently near $h$, the points $\pi(\vecm h'g)$ for 
$\vecm\in F$ are pairwise distinct.
Hence $\#(\scrP(\scrW,\Z^nhg)\cap U)=\# F$ and $\#(\scrP(\scrW,\Z^nh'g)\cap U)\geq\# F$ for every $h'$ near $h$.
Therefore in \eqref{KEYCONTLEMPF1}, the left hand side must be \textit{larger} than $\#F$, for all large $j$.
Hence for each large $j$ there is some $\vecm\in\Z^n\setminus F$ such that
$\vecm h_jg\in U\times\scrW$.
But for any compact $C\subset H_g$ the set $\cup_{h'\in C}(U\times\scrW)g^{-1}{h'}^{-1}$ is bounded and hence has
finite intersection with $\Z^n$.
Therefore there is a bounded number of possibilities for $\vecm$ as $j$ varies, and by passing to a subsequence we
may assume that $\vecm$ is independent of $j$.

Now for our fixed $\vecm\in\Z^n\setminus F$ we have $\vecm h_jg\in U\times\scrW$ for all $j$, 
but $\vecm h_jg\to\vecm  hg\notin U\times\scrW$ as $j\to\infty$;
this forces $\vecm  hg\in\partial(U\times\scrW)$,  
and it also implies that we cannot have $\vecm =\vecm_E$.  
But $\pi_{\intl}(\vecm  hg)\notin\partial\scrW$ since $h\notin S_1$,
and thus we must have $\pi(\vecm hg)\in\partial U$.
Note also that $\pi_{\intl}(\vecm  hg)$ cannot belong to the exterior of $\scrW$,
since then the same would hold for $\pi_{\intl}(\vecm  h_jg)$ for $j$ large, contradicting $\vecm h_jg\in U\times\scrW$.
Hence $\pi_{\intl}(\vecm hg)$ must belong to the interior of $\scrW$;
therefore $\pi(\vecm hg)\in\scrP^x=\scrP(\scrW,\Z^nhg)$.
This contradicts our assumption that $\scrP^x$ is disjoint from $\partial U$, and so the lemma is proved.
\end{proof}

\begin{lem}\label{KEYCONTLEM2}
For ${\mu}$-almost every $x\in X$, and for every bounded open set $U\subset\R^d$ with 
$\hatP^x\cap \partial U=\emptyset$,
there is an open set $\Omega\subset X$ with $x\in\Omega$ 
such that $\#(\hatP^{x'}\cap U)=\#(\hatP^x\cap U)$ for all $x'\in\Omega$.
\end{lem}
\begin{proof}
Let $\vecm_E$, $(\Z^n)'$, $S_1$ and $S_2$ be as in the proof of Lemma \ref{KEYCONTLEM}.
Also set
\begin{align*}
&S_3=\bigl\{h\in H_g\col \exists 
\vecm\in\Z^n,\: h'\in H_g\:\text{ satisfying }\:
\pi(\vecm hg)=\bn,\:\pi(\vecm h'g)\neq\bn\bigr\}
\\[5pt]
&S_4=\bigl\{h\in H_g\col \exists 
\vecm,\vecn\in\Z^n,\: h'\in H_g\:\text{ satisfying }\:
\dim\Span\{\pi(\vecn hg),\pi(\vecm hg)\}\leq1
\\
&\hspace{220pt}\text{and }\:\dim\Span\{\pi(\vecn h'g),\pi(\vecm h'g)\}=2\bigr\}.
\end{align*}
Using Lemma \ref{ALLORALMOSTNONELEM} and the fact that $\Z^n$ is countable,
we have ${\mu}(S_3)={\mu}(S_4)=0$.

Now let $h\in H_g\setminus(S_1\cup S_2\cup S_3\cup S_4)$ be given, set $x=\Gamma h\in X$, and 
let $U$ be an arbitrary bounded open subset of $\R^d$ with boundary disjoint from 
$\hatP^x=\hatP(\scrW,\Z^nhg)$.
Assume that there is a sequence $h_1,h_2,\ldots$ in $H_g$ tending to $h$ such that 
\begin{align}\label{KEYCONTLEM2PF1}
\#(\hatP(\scrW,\Z^nh_jg)\cap U)\neq\#(\hatP(\scrW,\Z^nhg)\cap U),\qquad\forall j.
\end{align}
We will show that this leads to a contradiction, and this will complete the proof of the lemma
(cf.\ the proof of Lemma \ref{KEYCONTLEM}).

As an initial reduction, let us note that we may assume $\scrP^x\cap \partial U=\emptyset$.
Indeed, recall that $\scrP^x$ is locally finite (cf.\ \cite[Prop.\ 3.1]{qc});
hence the set $A=\scrP^x\cap \partial U$ is certainly finite.
Also every point in $A$ is invisible in $\scrP^x$, since we are assuming $\hatP^x\cap\partial U=\emptyset$.
If $A\neq\emptyset$ then 
fix $r>0$ so small that $(\vecp+\scrB_{2r}^d)\cap\scrP^x=\{\vecp\}$ for each $\vecp\in A$, and set
$U'=U\cup(\cup_{\vecp\in A}(\vecp+\scrB_r^d))$ and $U''=U\setminus(\cup_{\vecp\in A}(\vecp+\overline{\scrB_r^d}))$.
These are bounded open sets satisfying $\#(\hatP^x\cap U')=\#(\hatP^x\cap U'')=\#(\hatP^x\cap U)$
and $\scrP^x\cap\partial U'=\scrP^x\cap\partial U''=\emptyset$.
For each $j$ we must have either $\#(\hatP(\scrW,\Z^nh_jg)\cap U')>\#(\hatP^x\cap U)$
or $\#(\hatP(\scrW,\Z^nh_jg)\cap U'')<\#(\hatP^x\cap U)$,
because of $U''\subset U\subset U'$ and \eqref{KEYCONTLEM2PF1}.
Hence after replacing $U$ by $U'$ or $U''$, and passing to a subsequence,
we are in a situation where \eqref{KEYCONTLEM2PF1} holds, and also  $\scrP^x\cap \partial U=\emptyset$.

Now take $F$ as in \eqref{KEYCONTLEMPF2}; it then follows from the proof of Lemma \ref{KEYCONTLEM} that
$\#(\scrP^x\cap U)=\# F$ and also 
$\#(\scrP(\scrW,\Z^nh_jg)\cap U)=\# F$ for every large $j$.
Hence \eqref{KEYCONTLEM2PF1} implies that for every large $j$ there is some $\vecm\in F$ such that
either $\pi(\vecm h_jg)$ is visible in $\scrP(\scrW,\Z^nh_jg)$ but $\pi(\vecm hg)$ is invisible in $\scrP^x$,
or the other way around.
Since $F$ is finite we may assume, by passing to a subsequence, that $\vecm$ is independent of $j$.

First assume that $\pi(\vecm hg)$ is invisible in $\scrP^x$
but $\pi(\vecm h_jg)$ is visible in $\scrP(\scrW,\Z^nh_jg)$ for every large $j$.
In particular then $\pi(\vecm h_jg)\neq\bn$ for large $j$, and since $h\notin S_3$ this implies $\pi(\vecm hg)\neq\bn$.
The invisibility of $\pi(\vecm hg)$ means that there exist $\vecn\in\Z^n$ and $0<t<1$ 
such that $\pi_\intl(\vecn hg)\in\scrW$ and $\pi(\vecn hg)=t\pi(\vecm hg)$.
Now $\pi_\intl(\vecn hg)\in\scrW$ and $h\notin S_1$ force $\pi_\intl(\vecn hg)\in\scrW^\circ$;
hence $\pi_\intl(\vecn h_jg)\in\scrW^\circ$ for all large $j$ and so $\pi(\vecn h_jg)\in\scrP(\scrW,\Z^nh_jg)$.
On the other hand $\dim\Span\{\pi(\vecn hg),\pi(\vecm hg)\}=1$ together with $h\notin S_4$ imply 
$\dim\Span\{\pi(\vecn h'g),\pi(\vecm h'g)\}\leq1$ for \textit{all} $h'\in H_g$.
Using also $h_j\to h$, $\pi(\vecm hg)\neq0$ and $0<t<1$, this implies
that for every large $j$ there is $0<t_j<1$ such that $\pi(\vecn h_jg)=t_j\pi(\vecm h_j g)$.
Hence $\pi(\vecm h_jg)$ is invisible in $\scrP(\scrW,\Z^nh_jg)$ for every large $j$,
contradicting our earlier assumption.


It remains to treat the case when $\pi(\vecm hg)$ is visible in $\scrP^x$
but $\pi(\vecm h_jg)$ is invisible in $\scrP(\scrW,\Z^nh_jg)$ for every large $j$.
Then for every large $j$ there exist $\vecn\in\Z^n$ and $0<t_j<1$ such that
$\pi_\intl(\vecn h_jg)\in\scrW$ and $\pi(\vecn h_jg)=t_j\pi(\vecm h_jg)$.
It is easily seen that there are only a finite number of possibilities for $\vecn$,
and hence by passing to a subsequence we may assume that $\vecn$ is independent of $j$.
Since $\pi(\vecm hg)$ is visible in $\scrP^x$ we have $\pi(\vecm hg)\neq\bn$;
hence also $\pi(\vecm h_jg)\neq\bn$ for all large $j$, and this forces $\vecn\neq\vecm$.
Also $\pi(\vecm h_jg)\to\pi(\vecm hg)\neq\bn$ and $t_j\pi(\vecm h_jg)=\pi(\vecn h_jg)\to\pi(\vecn hg)$ 
imply that $t=\lim_{j\to\infty}t_j\in[0,1]$ exists, and $\pi(\vecn hg)=t\pi(\vecm hg)$.
Using $h\notin S_1$ and $\pi_\intl(\vecn h_jg)\in\scrW$ 
it follows that also $\pi_\intl(\vecn hg)\in\scrW$ and so $\pi(\vecn hg)\in\scrP^x$.
Using $h\notin S_3$ and $\pi(\vecn h_jg)\neq\bn$ for $j$ large, it follows that $\pi(\vecn hg)\neq\bn$;
furthermore using $h\notin S_2$ we have $\pi(\vecn hg)\neq\pi(\vecm hg)$.
Hence $0<t<1$, and so $\pi(\vecm hg)$ is invisible in $\scrP^x$, contradicting our earlier assumption.
\end{proof}

\section{Upper bound on the density of visible points}\label{sec:eight}

We are now in position to prove an upper bound complementing Lemma \ref{LEM2}.

\begin{lem}\label{LEM3}
We have ${\displaystyle
\lim_{T\to\infty}\frac{\#(\hatP\cap \scrB_T^d)}{T^d}=\kappa_\scrP  C_\scrP v_d}$.
\end{lem}
\begin{proof}
For any $\scrP'\subset\R^d$, let us write $\widetilde\scrP'=\scrP'\setminus\hatP'$
for the set of \textit{in}visible points in $\scrP'$.
Define $F:X\to\Z_{\geq0}$ through
\begin{align}
F(x)=\liminf_{x'\to x}\:\#(\widetilde\scrP^{x'}\cap\scrB_1^d).
\end{align}
Then $F$ is lower semicontinuous by construction.
Hence by \cite[Thm.\ 4.1]{qc} and the Portmanteau theorem (cf.,\ e.g., \cite[Thm.\ 1.3.4(iv)]{wellner}),
\begin{align}\label{LEM3PF10}
\liminf_{R\to\infty}\int_{\SO(d)}F(\Gamma\varphi_g(k\Phi^{\log R}))\,dk\geq\int_{X}F\,d{\mu},
\end{align}
with
\begin{align}\label{PHITDEF}
\Phi^t=\matr{e^{-(d-1)t}}{\bn}{\trans\bn}{e^t1_{d-1}}\in\SL(d,\R).
\end{align}

Now in the left hand side of \eqref{LEM3PF10}, we use the fact that for any 
$x=\Gamma\varphi_g(T)$, $T\in\SL(d,\R)$, we have
\begin{align}
F(x)\leq \#(\widetilde\scrP^x\cap\scrB_1^d)=
\#(\widetilde\scrP(\scrW,\Z^n\varphi_g(T)g)\cap\scrB_1^d)
=\#(\widetilde\scrP\cap\scrB_1^d T^{-1}).
\end{align}
In the right hand side of \eqref{LEM3PF10} we note that if $x=\Gamma h$ has both the 
continuity properties described in Lemmata \ref{KEYCONTLEM} and \ref{KEYCONTLEM2},
and if furthermore $\scrP^x\cap\S_1^{d-1}=\emptyset$,
then in fact $F(x)=\#(\widetilde\scrP^x\cap\scrB_1^d)$.
But these conditions are fulfilled for ${\mu}$-almost all $x\in X$
(concerning $\scrP^x\cap\S_1^{d-1}=\emptyset$, use \cite[Thm.\ 1.5]{qc}).
Hence it follows from \eqref{LEM3PF10} that
\begin{align}\label{LEM3PF11}
\liminf_{R\to\infty}\int_{\SO(d)}\#\bigl(\widetilde\scrP\cap\scrB_1^d\Phi^{-\log R}k^{-1}\bigr)\,dk
\geq\int_X\#\bigl(\widetilde\scrP^x\cap\scrB_1^d\bigr)\,d{\mu}(x)
=(1-\kappa_\scrP)C_\scrP v_d,
\end{align}
where the last equality holds by Theorem \ref{VEECHTHM}. 

But exactly as in the proof of Theorem 5.1 in \cite{qc},
we have for any $R>1$
\begin{align}
\int_{\SO(d)}\#(\widetilde\scrP\cap\scrB_1\Phi^{-\log R}k^{-1})\,dk
=\sum_{\vecp\in\widetilde\scrP}A_R(\|\vecp\|)
=\int_{0}^\infty A_R(\tau)\,d\widetilde N(\tau)
=-\int_{0}^\infty \widetilde N(\tau)\,d A_R(\tau),
\end{align}
where
\begin{align}
\widetilde N(T)=\#(\widetilde\scrP\cap\scrB_T^d),
\end{align}
and $A_R$ is the continuous and decreasing function from $\R_{\geq0}$ to $[0,1]$ given by
$A_R(0)=1$ and 
\begin{align}
A_R(\tau)=\frac{\omega\bigl(\S_1^{d-1}\:\cap\:\tau^{-1}\scrB_1^d\Phi^{-\log R}\bigr)}{\omega(\S_1^{d-1})}
\qquad \text{for }\:\tau>0.
\end{align}
(Thus $A_R(\tau)=1$ for $0\leq\tau\leq R^{-1}$ and $A_R(\tau)=0$ for $\tau\geq R^{d-1}$.)
Hence \eqref{LEM3PF11} says that
\begin{align}\label{LEM3PF12}
\liminf_{R\to\infty}\int_{0}^\infty \widetilde N(\tau)\,\bigl(-d A_R(\tau)\bigr)
\geq C':=(1-\kappa_\scrP)C_\scrP v_d.
\end{align}

In view of \eqref{density000} and Lemma \ref{LEM2} (with $\scrD=\scrB_1^d$), 
the statement of the present lemma is equivalent with 
$\liminf_{\tau\to\infty}\tau^{-d}\widetilde N(\tau)\geq C'$.
Assume that this is \textit{false}.
Then there is some $\eta>0$ and a sequence $1<\tau_1<\tau_2<\cdots$ with $\tau_j\to\infty$
such that $\widetilde N(\tau_j)<(1-\eta) C' \tau_j^d$ for all $j$.
Using the fact that $\widetilde N(\tau)$ is an increasing function of $\tau$ we see that by shrinking $\eta>0$ if necessary,
we may actually assume that $\widetilde N(\tau)<(1-\eta) C' \tau^d$ for all $\tau\in[(1-\eta)\tau_j,\tau_j]$ and all $j$.
By Lemma \ref{LEM2} and \eqref{density000} we have
$\limsup_{\tau\to\infty}\tau^{-d}\widetilde N(\tau)\leq C'$;
thus for any given $\ve>0$ there is some $\tau_0>0$ such that $\widetilde N(\tau)\leq(1+\ve)C'\tau^d$ for all $\tau\geq \tau_0$.
Now for any $j$ with $(1-\eta)\tau_j>\tau_0$, and any $R>\tau_j^{1/(d-1)}$:
\begin{align}
\int_{0}^\infty \widetilde N(\tau)\,\bigl(-d A_R(\tau)\bigr)
\leq \int_0^{\tau_0}\widetilde N(\tau)(-dA_R(\tau)) &+(1+\ve)C'\int_{\tau_0}^{R^{d-1}}\tau^d\,(-dA_R(\tau))
\\\notag
&-(\ve+\eta) C'\int_{(1-\eta)\tau_j}^{\tau_j}\tau^d\,(-dA_R(\tau)).
\end{align}
Here the sum of the first two terms tends to $(1+\ve)C'$ as $R\to\infty$,
as in \cite[(5.11)-(5.13)]{qc}.
Furthermore, if we choose $R=(2\tau_j)^{1/(d-1)}$ and let $j\to\infty$ then 
\begin{align}
\int_{(1-\eta)\tau_j}^{\tau_j}\tau^d\,(-dA_R(\tau))
=\frac d{\omega(\S_1^{d-1})}\vol\Bigl(\scrB_1^d\Phi^{-\log R}\cap\scrB_{\frac12R^{d-1}}^d\setminus
\scrB_{\frac12(1-\eta)R^{d-1}}^d\Bigr)
\\\notag
\to\frac{2v_{d-1}}{v_d}\int_{(1-\eta)/2}^{1/2}(1-x^2)^{(d-1)/2}\,dx.
\end{align}
Hence we conclude that there is a constant $c(\eta)>0$, independent of $\ve$, such that
\begin{align}
\liminf_{R\to\infty}\int_{0}^\infty \widetilde N(\tau)\,\bigl(-d A_R(\tau)\bigr)\leq(1+\ve-c(\eta))C'.
\end{align}
Letting now $\ve\to0$ we run into a contradiction against \eqref{LEM3PF12}.
This concludes the proof of the lemma.
\end{proof}

\section{Proof of Theorem \ref{thm:main1}}\label{sec:nine}

Combining Lemma \ref{LEM2} and Lemma \ref{LEM3} we can now complete the proof of Theorem \ref{thm:main1}.
First let $\fU$, $\scrD$ be as in Lemma \ref{LEM2}.
Then by Lemma \ref{LEM2} applied to $\S_1^{d-1}\setminus\fU$,
\begin{align}
\liminf_{T\to\infty}\frac{\#(\hatP\cap \scrB_T^d\setminus T\scrD)}{T^d}\geq\kappa_\scrP  C_\scrP 
\bigl(v_d-\vol(\scrD)\bigr).
\end{align}
Combining this with Lemma \ref{LEM3} we get
\begin{align}
\limsup_{T\to\infty}\frac{\#(\hatP\cap T\scrD)}{T^d}=
\limsup_{T\to\infty}\biggl(\frac{\#(\hatP\cap\scrB_T^d)}{T^d}-
\frac{\#(\hatP\cap \scrB_T^d\setminus T\scrD)}{T^d}\biggr)
\leq\kappa_\scrP C_\scrP\vol(\scrD).
\end{align}
Combining this with Lemma \ref{LEM2} (applied to $\fU$ itself) we conclude 
\begin{align}\label{KPPF1}
\lim_{T\to\infty}\frac{\#(\hatP\cap T\scrD)}{T^d}=\kappa_\scrP C_\scrP\vol(\scrD).
\end{align}
By a scaling and subtraction argument it follows that \eqref{KPPF1} is true more generally for any 
$\scrD\in\scrF$, where $\scrF$ is the family of sets of the form 
$\scrD=\{\vecv\in\R^d\col r_1\leq \|\vecv\|<r_2,\:\vecv\in\|\vecv\|\fU\}$,
for any $0\leq r_1<r_2$ and 
any $\fU\subset\S_1^{d-1}$ with $\omega(\partial\fU)=0$.

Now let $\scrD$ be an arbitrary subset of $\R^d$ with boundary of measure zero.
Note that the validity of \eqref{KPPF1} does not change if we replace $\scrD$ by $\scrD\cup\{\bn\}$
or by $\scrD\setminus\{\bn\}$. 
The proof of Theorem \ref{VEECHTHM} is now completed by approximating
$\scrD\cup\{\bn\}$ from above and $\scrD\setminus\{\bn\}$ from below by finite unions of sets in $\scrF$.

\section{Proof of Theorem \ref{thm:main2}}\label{sec:ten}

Recall that \eqref{eq:main21} was proved in \cite[Thm.\ A.1]{qc}
and we have proved \eqref{eq:main23} and \eqref{eq:main24} in Section~\ref{sec:five}.
Also the continuity of $E(r,\sigma,\scrP)$ and $E(r,\sigma,\hatP)$ with respect to $\sigma$
is immediate from \eqref{eq01}, \eqref{eq02} combined with Theorem \ref{VEECHTHM}.
Hence it remains to prove \eqref{eq:main22}.


Thus let $\lambda$ be a Borel probability measure on $\S_1^{d-1}$ which is absolutely continuous with respect to $\omega$,
and let $\sigma>0$ and $r\in\ZZ_{\geq 0}$.
Let us fix, once and for all, a map $K:\S_1^{d-1}\to\SO(d)$ such that
$\vecv K(\vecv)=\vece_1=(1,0,\ldots,0)$ for all $\vecv\in\S_1^{d-1}$;
we assume that $K$ is smooth when restricted to $\S_1^{d-1}$ minus
one point, cf.~\cite[Footnote 3, p.\ 1968]{partI}.
Recall the definitions of $\fC(\sigma)$ and $\Phi^t$ in \eqref{FCCSDEF} and \eqref{PHITDEF}.

On verifies that if $\sigma',\sigma'',\alpha$ are any fixed numbers satisfying $0<\sigma'<\sigma<\sigma''$
and $\sigma'/\sigma<\alpha <1$,
then for any $\vecv\in\S_1^{d-1}$ and all sufficiently large $T$,
the set of $\vecy\in\scrB_T^d\setminus\{\bn\}$ satisfying 
$\|\vecy\|^{-1}\vecy\in\fD_T(\kappa_\scrP^{-1}\sigma,\vecv)$ 
is \textit{contained} in $\fC(\kappa_\scrP^{-1}\sigma'')\Phi^{-(\log T)/(d-1)}K(\vecv)^{-1}$,
and \textit{contains} $\fC(\kappa_\scrP^{-1}\sigma')\Phi^{-(\log(\alpha T))/(d-1)}K(\vecv)^{-1}$.
It follows that 
\begin{align}\notag
\lambda\bigl(\bigl\{\vecv\in\S_1^{d-1}\col \#\bigl(\hatP\cap 
\fC(\kappa_\scrP^{-1}\sigma'')\Phi^{-(\log T)/(d-1)}K(\vecv)^{-1}\bigr)\leq r\bigr\}\bigr)
\hspace{80pt}
\\\label{MAINVISTHMPF2}
\leq\lambda\bigl(\bigl\{\vecv\in\S_1^{d-1}\col\scrN_{T}(\sigma,\vecv,\hatP)\leq r\bigr\}\bigr)
\hspace{190pt}
\\\notag
\leq\lambda\bigl(\bigl\{\vecv\in\S_1^{d-1}\col \#\bigl(\hatP\cap 
\fC(\kappa_\scrP^{-1}\sigma')\Phi^{-(\log(\alpha T))/(d-1)}K(\vecv)^{-1}\bigr)\leq r\bigr\}\bigr).
\end{align}

Recalling the definition of $\scrP=\scrP(\scrW,\Z^n g)$ we see that
$\hatP A=\hatP(\scrW,\Z^n\varphi_g(A)g)$ for any $A\in\SL(d,\R)$.
Hence if we define
\begin{align}
\scrE(\sigma,r)=\bigl\{x\in X\col 
\#\bigl(\hatP^x 
\cap\fC(\kappa_\scrP^{-1}\sigma)\bigr)\leq r\bigr\},
\end{align}
then the left hand side in \eqref{MAINVISTHMPF2} equals
\begin{align}
\lambda\bigl(\bigl\{\vecv\in\S_1^{d-1}\col 
\Gamma \varphi_g\bigl(K(\vecv)\Phi^{(\log T)/(d-1)}\bigr)\in \scrE(\sigma'',r) 
\bigr\}\bigr)
\end{align}
Hence by \cite[Thm.\ 4.1]{qc} and the Portmanteau theorem:
\begin{align}\label{MAINTHMPF12}
\liminf_{T\to\infty}\lambda\bigl(\bigl\{\vecv\in\S_1^{d-1}\col\scrN_{T}(\sigma,\vecv,\hatP)\leq r\bigr\}\bigr)
\geq \mu\bigl(\scrE(\sigma'',r)^\circ\bigr)
=\mu\bigl(\scrE(\sigma'',r)\bigr).
\end{align}
Here the last equality is proved by using Lemma \ref{KEYCONTLEM2} with $U=\fC(\kappa_\scrP^{-1}\sigma'')$,
and noticing that Theorem \ref{VEECHTHM} implies that
$\hatP^x\cap\partial U=\emptyset$ for $\mu$-almost all $x\in X$.
Similarly, using the right relation in \eqref{MAINVISTHMPF2}, we obtain
\begin{align}\label{MAINTHMPF13}
\limsup_{T\to\infty}\lambda\bigl(\bigl\{\vecv\in\S_1^{d-1}\col\scrN_{T}(\sigma,\vecv,\hatP)\leq r\bigr\}\bigr)
\leq\mu\bigl(\overline{\scrE(\sigma',r)}\bigr)=\mu\bigl(\scrE(\sigma',r)\bigr).
\end{align}

Note that $\scrE(\sigma'',r)\subset\scrE(\sigma,r)\subset\scrE(\sigma',r)$, since
$\fC(\kappa_\scrP^{-1}\sigma'')\supset\fC(\kappa_\scrP^{-1}\sigma)\supset\fC(\kappa_\scrP^{-1}\sigma')$.
Also, if $x$ lies in $\scrE(\sigma,r)$ but not in $\scrE(\sigma'',r)$,
then $\hatP^x$ must have some point in $\fC(\kappa_\scrP^{-1}\sigma'')\setminus\fC(\kappa_\scrP^{-1}\sigma)$,
and so by Theorem \ref{VEECHTHM},
\begin{align}\label{MAINTHMPF14}
\mu\bigl(\scrE(\sigma,r)\bigr) - \mu\bigl(\scrE(\sigma'',r)\bigr)
\leq\kappa_\scrP C_\scrP \vol\bigl(\fC(\kappa_\scrP^{-1}\sigma'')\setminus\fC(\kappa_\scrP^{-1}\sigma)\bigr).
\end{align}
Similarly
\begin{align}\label{MAINTHMPF15}
\mu\bigl(\scrE(\sigma',r)\bigr) - \mu\bigl(\scrE(\sigma,r)\bigr)
\leq\kappa_\scrP C_\scrP \vol\bigl(\fC(\kappa_\scrP^{-1}\sigma)\setminus\fC(\kappa_\scrP^{-1}\sigma')\bigr).
\end{align}
Now by taking $\sigma',\sigma''$ sufficiently near $\sigma$,
the right hand sides of \eqref{MAINTHMPF14} and \eqref{MAINTHMPF15} can be made as small as we like.
Hence from \eqref{MAINTHMPF12} and \eqref{MAINTHMPF13} we obtain in fact
\begin{align}\label{MAINTHMPF16}
\lim_{T\to\infty}\lambda\bigl(\bigl\{\vecv\in\S_1^{d-1}\col\scrN_{T}(\sigma,\vecv,\hatP)\leq r\bigr\}\bigr)
=\mu\bigl(\scrE(\sigma,r)\bigr)
=\mu\bigl(\bigl\{x\in X\col \#\bigl(\hatP^x 
\cap\fC(\kappa_\scrP^{-1}\sigma)\bigr)\leq r\bigr\}\bigr).
\end{align}
Note here that the right hand side is the same as $\sum_{r'=0}^rE(r,\sigma,\widehat \scrP)$; cf.\ \eqref{eq02}.
Hence since \eqref{MAINTHMPF16} has been proved for arbitrary $r\geq0$,
also \eqref{eq:main22} holds for arbitrary $r\geq0$, and we are done.

\section{Proof of Corollary \ref{GAPDISTRCOR}}\label{sec:eleven}

It follows from Theorem \ref{thm:main2} and a general statistical argument (cf. e.g. \cite{Marklof07})
that if we define $F(0)=0$ and 
\begin{align}\label{GAPDISTRCORPF1}
F(s)=-\frac d{ds} E(0,s,\hatP),
\end{align}
then the limit relation \eqref{GAPDISTRCORRES1} holds at each point $s\geq0$ where $F(s)$ is defined.
In fact the function $s\mapsto E(0,s,\hatP)$ is convex;
hence $F(s)$ exists for all $s>0$ except at most a countable number of points,
and is continuous at each point where it exists.
Also $F(s)$ is decreasing, and satisfies $\lim_{s\to0^+}F(s)=1=F(0)$ (cf.\ \eqref{eq:main24}) and $\lim_{s\to\infty}F(s)=0$.
Note also that \eqref{GAPDISTRCORRES2} is an immediate consequence of \eqref{GAPDISTRCORRES1},
the definition of $\widehat\xi_{T,j}$
and the fact that $N(T)\sim\kappa_\scrP^{-1}\widehat N(T)$ as $T\to\infty$
(cf.\ Theorem \ref{thm:main1} and \eqref{rel-dens}).

It now only remains to prove that $F(s)$ is continuous,
or equivalently that the derivative in \eqref{GAPDISTRCORPF1} exists for every $s>0$.
Assume the contrary, and let $s_0>0$ be a point where the derivative does \textit{not} exist.
By convexity, both the left and right derivative exist at $s_0$; thus
\begin{align}\label{GAPDISTRCORPF2}
-\lim_{s\to s_0^-}\frac{E(0,s_0,\hatP)-E(0,s,\hatP)}{s_0-s}
>
-\lim_{s\to s_0^+}\frac{E(0,s,\hatP)-E(0,s_0,\hatP)}{s-s_0}
\geq0.
\end{align}

In dimension $d=2$,
using the fact that the point process $x\mapsto\hatP^x$ is invariant under $\smatr 1r01\in\SL(2,\R)$ for any $r\in\R$,
it follows that the formula \eqref{eq01} holds with $\fC(\sigma)$ replaced by $\fC(a,a+\sigma)$ for any $a\in\R$,
where
\begin{align*}
\fC(a_1,a_2)=\Bigl\{\vecy=(y_1,y_2)\in\R^2\col 0<y_1<1,\:\frac2{\kappa_\scrP C_\scrP}a_1 y_1<y_2
<\frac2{\kappa_\scrP C_\scrP}a_2 y_1\Bigr\}
\end{align*}
In particular, for any $0<s<s'$ and $a\in\R$,
\begin{align}\label{GAPDISTRCORPF3}
E(0,s,\hatP)-E(0,s',\hatP)=\mu\bigl(\bigl\{x\in X\col \hatP^x\cap\fC(a,a+s)=\emptyset,\:
\hatP^x\cap\fC(a,a+s')\neq\emptyset\bigr\}\bigr).
\end{align}

For given $x\in X$, we order the numbers 
\begin{align*}
\frac{\kappa_\scrP C_\scrP}2\cdot\frac{y_2}{y_1}
\qquad\text{for }\:\vecy=(y_1,y_2)\in\hatP^x\cap ((0,1)\times\R_{>0})
\end{align*}
as $0<\lambda_{x,1}<\lambda_{x,2}<\ldots$.
We also set $\lambda_{x,0}=0$.
Taking $s'=s_0>s$ in \eqref{GAPDISTRCORPF3}, 
integrating over $a\in(0,a_0)$ for some fixed $a_0>0$, and using Fubini's Theorem,
we obtain
\begin{align*}
a_0\bigl(E(0,s,\hatP)-E(0,s_0,\hatP)\bigr)\leq
\int_X (s_0-s)\#\bigl\{j\geq0\col \lambda_{x,j+1}-\lambda_{x,j}>s,\: \lambda_{x,j+1}<a_0+s_0\bigr\}\,d\mu(x).
\end{align*}
Hence 
\begin{align}\label{GAPDISTRCORPF4}
-a_0\lim_{s\to s_0^-}\frac{E(0,s_0,\hatP)-E(0,s,\hatP)}{s_0-s}
\leq\int_X \#\bigl\{j\geq0\col \lambda_{x,j+1}-\lambda_{x,j}\geq s_0,\: \lambda_{x,j+1}<a_0+s_0\bigr\}\,d\mu(x).
\end{align}
Similarly, replacing $s$ by $s_0$ and $s'$ by $s$ in \eqref{GAPDISTRCORPF3}, we obtain
\begin{align}\label{GAPDISTRCORPF5}
-a_0\lim_{s\to s_0^+}\frac{E(0,s,\hatP)-E(0,s_0,\hatP)}{s-s_0}
\geq\int_X \#\bigl\{j\geq0\col \lambda_{x,j+1}-\lambda_{x,j}>s_0,\: \lambda_{x,j+1}<a_0+s_0\bigr\}\,d\mu(x).
\end{align}

It follows from \eqref{GAPDISTRCORPF2}, \eqref{GAPDISTRCORPF4} and \eqref{GAPDISTRCORPF5} that 
there is a set $A\subset X$ with $\mu(A)>0$ such that for every $x\in A$,
there is some $j\geq0$ such that
$\lambda_{x,j+1}-\lambda_{x,j}=s_0$ and $\lambda_{x,j}<a_0$.
Note that $\lambda_{x,1}\neq s_0$ for $\mu$-almost all $x\in X$,
by Theorem \ref{VEECHTHM} applied with $f$ as the characteristic function of the line
$y_2=s_0\frac2{\kappa_\scrP C_\scrP} y_1$ in $\R^2$.
Hence after removing a null set from $A$, we have for each $x\in A$
that $\hatP^x$ contains a pair of points $\vecy=(y_1,y_2)$ and $\vecy'=(y_1',y_2')$ satisfying 
\begin{align*}
0<y_1,y_1'<1,\qquad
\frac{y_2'}{y_1'}-\frac{y_2}{y_1}=\frac2{\kappa_\scrP C_\scrP}s_0,\qquad
0<\frac{y_2}{y_1}<\frac2{\kappa_\scrP C_\scrP}a_0.
\end{align*}
However this is easily seen to violate the $\SL(2,\R)$-invariance of the point process $x\mapsto\hatP^x$.
For example, for each $\frac12\leq\lambda\leq1$, because of the invariance under 
$\smatr{\sqrt\lambda}00{1/\sqrt\lambda}$,
there is a subset $A_\lambda\subset X$ with $\mu(A_\lambda)=\mu(A)>0$ such that for each $x\in A_\lambda$,
$\hatP^x$ contains a pair of points $\vecy=(y_1,y_2)$ and $\vecy'=(y_1',y_2')$ satisfying 
\begin{align*}
0<y_1,y_1'<\sqrt\lambda,\qquad
\frac{y_2'}{y_1'}-\frac{y_2}{y_1}=\frac2{\kappa_\scrP C_\scrP}\frac{s_0}{\lambda},\qquad
0<\frac{y_2}{y_1}<\frac2{\kappa_\scrP C_\scrP}\frac{a_0}{\lambda}.
\end{align*}
Let $R$ be the rectangle $(0,1)\times(0,\frac4{\kappa_\scrP C_\scrP}(a_0+s_0))$ in $\R^2$.
By taking $N$ sufficiently large we can ensure that the set $X_{R,N}:=\{x\in X\col \#(\hatP^x\cap R)\leq N\}$
has measure $\mu(X_{R,N})\geq1-\frac12\mu(A)$.
It follows that $\mu(A_\lambda\cap X_{R,N})\geq\frac12\mu(A)$ for each $\frac12\leq\lambda\leq1$,
and so 
if $\Lambda$ is any infinite subset of $[\frac12,1]$
then the integral $\int_{X_{R,N}}\sum_{\lambda\in\Lambda}I(x\in A_\lambda)\,d\mu(x)$ is infinite.
On the other hand the definition of $X_{R,N}$ implies that 
$\sum_{\lambda\in\Lambda}I(x\in A_\lambda)\leq\binom N2$ for each $x\in X_{R,N}$.

We have thus reached a contradiction, and we conclude that \eqref{GAPDISTRCORPF2} cannot hold,
i.e.\ $F(s)$ is continuous for all $s\geq0$.

\section{Vanishing near zero of the gap distribution}\label{sec:twelwe}

The gap distribution obtained in Corollary \ref{GAPDISTRCOR} may sometimes vanish near zero.
This phenomenon was noted numerically in \cite{BGHJ} in several examples.
In the case when $\scrP$ is a \textit{lattice}, this vanishing is well understood;
cf.\ \cite{Boca00}, \cite{partI}. We recall that in this case the limit distribution, and hence the gap size, is independent of the choice of lattice. 

Let $\scrP=\scrP(\scrW,\scrL)$ be a regular cut-and-project set.
We define $m_{\hatP}\geq0$ to be the supremum of all $\sigma\geq0$ with the property that 
$\#(\widehat\scrP^x\cap\fC(\kappa_\scrP^{-1}\sigma))\leq 1$ for ($\mu$-)almost all $x\in X$.
Then the computation in \eqref{LEM1pf3} 
(together with \eqref{E0BASICREL}) shows that
\begin{align}\label{MHATPPROPERTY1}
E(0,\sigma,\hatP)\quad\begin{cases}=1-\sigma &\text{when }\: 0\leq\sigma\leq m_{\hatP}
\\
>1-\sigma &\text{when }\: \sigma>m_{\hatP}.
\end{cases}
\end{align}
We note that if $d\geq3$ then $m_{\hatP}=0$,
because of the $\SL(d,\R)$-invariance and the fact that 
$\SL(d,\R)$ acts transitively on pairs of non-proportional vectors in $\R^d\setminus\{\bn\}$ when $d\geq3$.

Let us now assume $d=2$. Note that by \eqref{MHATPPROPERTY1} and the discussion at the beginning of Sec.\ \ref{sec:eleven},
the function $F$ in Corollary \ref{GAPDISTRCOR} satisfies
\begin{align*}
F(s)\quad\begin{cases} =1&\text{if }\: 0\leq s\leq m_{\hatP}
\\
<1&\text{if }\: s>m_{\hatP}.
\end{cases}
\end{align*}
In other words, $m_{\hatP}$ is the largest number with the property that
the limiting gap distribution obtained in Corollary \ref{GAPDISTRCOR} 
is supported on the interval $[m_{\hatP},\infty)$.
In particular, the support of the limiting gap distribution is separated from $0$ if and only if $m_{\hatP}>0$.

Let us also note that if $d=2$, $m\geq1$, and $\scrL$ is a ``generic'' lattice or affine lattice,
so that either $H_g=\SL(n,\R)$ or $H_g=G=\ASL(n,\R)$,
then we have $m_{\hatP}=0$, again using the transitivity of the action of $\SL(n,\R)$ on pairs of non-proportional vectors in
$\R^n\setminus\{\bn\}$ for $n\geq3$.

\vspace{5pt}

On the other hand, we will now recall (for general $d$) 
a standard construction of cut-and-project sets using the geometric representation of algebraic numbers \cite{Borevich},
which can be used to produce several of the most well-known quasicrystals, cf.~\cite{Baake13,meyer,meyer95,Moody97,Moody00,Pleasants03}, 
We will see that in special cases with $d=2$, this construction leads to quasicrystals for which $m_\hatP>0$.

We follow \cite[Sec.\ 2.2]{qc}.
Let $K$ be a totally real number field of degree $N\geq2$ over $\Q$,
let $\scrO_K$ be its subring of algebraic integers,
and let $\pi_1,\ldots,\pi_N$  be the distinct embeddings of $K$ into $\R$.
We will always view $K$ as a subset of $\R$ via $\pi_1$; in other words we agree that $\pi_1$ is the identity map.
Fix $d\geq1$ and set $n=dN$.
By abuse of notation we write $\pi_j$ also for the coordinate-wise embedding of $K^d$ into $\R^d$,
and for the entry-wise embedding of $M_d(K)$ (the algebra of $d\times d$ matrices with entries in $K$) into $M_d(\R)$.
Let $\scrL$  %
be the lattice in $\R^n=(\R^d)^N$ given by
\begin{align}\label{NUMFIELDLATTICE}
\scrL=\scrL_K^d:   %
=\Bigl\{(\vecx,\pi_2(\vecx),\ldots,\pi_N(\vecx))\col \vecx\in\scrO_K^d\Bigr\}.
\end{align}
As usual we set $m=n-d=(N-1)d$, let $\pi$ and $\pi_\intl$ be the projections of
$\R^n=(\R^d)^N=\R^d\times\R^m$ onto the first $d$ and last $m$ coordinates.
It follows from \cite[Cor.\ 2 in Ch.\ IV-2]{Weil} that $\pi_\intl(\scrL)$ is dense in $\R^m$,
i.e.\ we have $\scrA=\R^m$ and $\scrV=\R^n$ in the present situation.
Hence the window $\scrW$ should be taken as a subset of $\R^m$, and we consider
the cut-and-project set $\scrP(\scrW,\scrL)\subset\R^d$.

Choose $\delta>0$ and $g\in\SL(n,\R)$ such that
\begin{align}\label{SCRLDEF}
\scrL=\delta^{1/n}\Z^ng.
\end{align}
In fact
\begin{align}\label{DETOFLK}
\delta=|D_K|^{d/2},
\end{align}
where $D_K$ is the discriminant of $K$; %
cf., e.g., \cite[Ch.\ V.2, Lemma 2]{lang}.
As proved in \cite[Sec.\ 2.2.1]{qc}, in this situation we have
\begin{align}\label{HGNUMBERFIELDCASE}
H_g=g\SL(d,\R)^Ng^{-1};
\end{align}
where $\SL(d,\R)^N$ is embedded as a subgroup of $G=\ASL(n,\R)$ through
\begin{align}\label{SLdRNIMB}
(A_1,\ldots,A_N)\mapsto\Bigl(\diag[A_1,\ldots,A_N],\bn\Bigr), 
\end{align}
where $\diag[A_1,\ldots,A_N]$ is the block matrix whose diagonal blocks are $A_1,\ldots,A_N$ in this order,
and all other blocks vanish.

\begin{lem}\label{MPFROMBELOWLEM}
Let $\scrP=\scrP(\scrW,\scrL)$ be a regular cut-and-project set with $\scrL$ as in \eqref{NUMFIELDLATTICE},
and with $d=N=2$ (thus $K$ is a real quadratic number field).
Let $\ve>1$ be the fundamental unit of $\scrO_K$, and set $R=\sup\{\|\vecw\|\col\vecw\in\scrW\}$.
Then
\begin{align}\label{MPFROMBELOWLEMRES}
m_\hatP\geq\frac{\kappa_\scrP C_\scrP \delta}{(\ve^2+\ve^{-2})^2R^2}.
\end{align}
\end{lem}
\begin{proof}
Let $\sigma>0$ and $x\in X$ be given and assume that
$\#(\widehat\scrP^x\cap\fC(\kappa_\scrP^{-1}\sigma))\geq2$.
It suffices to prove that we must then have
$\sigma\geq\frac{\kappa_\scrP C_\scrP \delta}{(\ve^2+\ve^{-2})^2R^2}.$
The area of $\fC(\kappa_\scrP^{-1}\sigma)$ equals $r^2$ where $r:=\sqrt{\frac{\sigma}{\kappa_\scrP C_\scrP}}$;
hence there is some $A\in\SL(2,\R)$ 
which maps $\fC(\kappa_\scrP^{-1}\sigma)$ to the open triangle
$\fC_r:=\{\vecx\in\R^2\col 0<x_1<r, \: |x_2|<x_1\}$.
Take $(A_1,A_2)\in\SL(2,\R)^2$ (embedded in $G$ as in \eqref{SLdRNIMB}) so that $x=\Gamma g(A_1,A_2)g^{-1}$.
Set $\widetilde A=(A_1A,A_2)$; then 
$\scrP^xA=\scrP(\scrW,\scrL\widetilde A)$. 
We set $\gamma=\diag[\ve^{-k},\ve^{-k},\ve^{k},\ve^{k}]\in\SL(4,\R)$, where $k$ is an integer which we will choose below.
Then $\scrL\widetilde A=\scrL\gamma\widetilde A=\scrL\widetilde A\gamma$,
by \eqref{NUMFIELDLATTICE} and since $\widetilde A$ is block diagonal. Hence
\begin{align*}
\scrP^xA=\scrP(\scrW,\scrL\widetilde A)=\scrP(\scrW,\scrL\widetilde A\gamma)
=\ve^{-k}\scrP(\ve^{-k}\scrW,\scrL\widetilde A).
\end{align*}
Now $\#(\widehat\scrP^xA\cap \fC_r)\geq2$ and thus
$\scrL\widetilde A$ contains two points in $(\ve^{k}\fC_r)\times(\ve^{-k}\scrW)$ which have non-proportional images 
under $\pi$ (the projection onto the physical space $\R^2$). 
In other words, there exist $\vecx,\vecx'\in\scrO_K^2\subset\R^2$ 
which are linearly independent over $\R$ (thus also over $K$) such that
$\vecb_1=(\vecx,\overline{\vecx})\widetilde A$ and $\vecb_2=(\vecx',\overline{\vecx'})\widetilde A$
lie in $(\ve^{k}\fC_r)\times(\ve^{-k}\scrW)$.
Here we write $\vecx\mapsto\overline\vecx$ for the nontrivial automorphism of $K$.
It follows that also 
$\vecb_3=(\ve\vecx,\overline{\ve\vecx})\widetilde A$ and $\vecb_4=(\ve\vecx',\overline{\ve\vecx'})\widetilde A$
lie in $(\ve^{k+1}\fC_r)\times(\ve^{-k-1}\scrW)$.
However the four vectors $(\vecx,\overline{\vecx})$, $(\vecx',\overline{\vecx'})$,
$(\ve\vecx,\overline{\ve\vecx})$, $(\ve\vecx',\overline{\ve\vecx'})$ lie in $\scrL$ and form a $K$-linear basis of $K^4$.
Hence $\vecb_1,\vecb_2,\vecb_3,\vecb_4$ lie in $\scrL\widetilde A$ and are linearly independent over $\R$.
However $\|\vecb_j\|<r'$ for $j=1,2,3,4$, where
\begin{align*}
r'=\max\left(\sqrt{(\ve^k r)^2+(\ve^{-k}R)^2},\sqrt{(\ve^{k+1}r)^2+(\ve^{-k-1}R)^2}\right),
\end{align*}
and thus $\delta$, the covolume of $\scrL\widetilde A$, must be less than ${r'}^4$.
Now choose $k$ so as to minimize $r'$.
Then $r'\leq\sqrt{\ve^2+\ve^{-2}}\sqrt{Rr}$,
and combining this with $\delta<{r'}^4$ and 
$r=\sqrt{\frac{\sigma}{\kappa_\scrP C_\scrP}}$ we obtain
$\sigma>\frac{\kappa_\scrP C_\scrP \delta}{(\ve^2+\ve^{-2})^2R^2}$,
as desired.
\end{proof}

Let us make some further observations in this vein.
First, note the general relation
\begin{align*}
\scrP(\scrW,q^{-1}\scrL)=q^{-1}\scrP(q\scrW,\scrL),\qquad\text{$\forall$ $q>0$ (real).}
\end{align*}
Using this relation with $q$ an appropriate positive integer, it is clear that if $\scrL$ is any lattice in
$\R^n$ such that the cut-and-project set $\scrP=\scrP(\scrW,\scrL)$ satisfies $m_{\hatP}>0$
for every admissible window set $\scrW$
(for example this holds when $\scrL$ is as in Lemma \ref{MPFROMBELOWLEM}),
then $m_{\hatP}>0$ \textit{also} holds for any cut-and-project set obtained from $\scrP(\scrW,\scrL)$ by the
``union of rational translates'' construction in \cite[Sec.\ 2.3.1]{qc}.
Furthermore, the property of having $m_{\hatP}>0$ is also, obviously,
preserved under ``passing to a sublattice'' as in \cite[Sec.\ 2.4]{qc}.
In particular, by \cite[Sec.\ 2.5]{qc} and Remark \ref{PENROSEREMARK} below,
we have $m_\hatP>0$ for any $\scrP$ associated with a rhombic Penrose tiling.

\begin{remark}\label{PENROSEREMARK}
If we wish to reproduce the vertex set of an \textit{arbitrary} rhombic Penrose tiling (RPT)
as a cut-and-project set within the present framwork, we also need to consider the case of so-called \textit{singular} vectors $\vecgamma$, as explained by de Bruijn \cite{bruijn} (we use the same notation as in \cite[Sec.\ 2.5]{qc}).
In this case there are either 2 or 10 distinct RPT's associated to $\vecgamma$,
and by \cite[Sec.\ 12]{bruijn} (carried over to our notation), 
the vertex set of any of these can be expressed as
\begin{align*}
\bigl\{\pi(\vecy)\col\vecy\in\scrL,\:\:[\exists M>0\col\forall m>M\col \pi_\intl(\vecy)\in\scrW(\vecgamma^{(m)})]\bigr\},
\end{align*}
where $\vecgamma^{(m)}$ is an appropriate sequence of regular vectors tending to $\vecgamma$ as $m\to\infty$,
and we write $\scrW=\scrW(\vecgamma)$ for the open window set defined in \cite[(2.25)]{qc}.
In other words, the vertex set of the RPT equals $\scrP(\widetilde\scrW,\scrL)$,
where $\widetilde\scrW:=\{\vecv\in\scrA\col[\exists M>0\col\forall m>M\col\vecv\in\scrW(\vecgamma^{(m)})]\}$.
Note that $\widetilde\scrW$ is the union of the open set $\scrW(\vecgamma)$ and part of its boundary.
In particular $\partial\widetilde\scrW=\partial\scrW(\vecgamma)$ has measure zero with respect to $\mu_\scrA$.
Hence the vertex set is again a regular cut-and-project set,
and the previous discussion leading to $m_\hatP>0$ applies.
\end{remark}

\begin{remark}
We do not expect the lower bound in Lemma \ref{MPFROMBELOWLEM} to be sharp,
and the argument which we gave regarding the construction in \cite[Sec.\ 2.3.1]{qc} certainly does not lead to
a sharp bound.
It would be interesting to try to determine the \textit{exact} value of $m_\hatP$ for the Penrose tiling,
and also for some of the cases discussed in \cite{BGHJ}.
\end{remark}

It is interesting to note that for a large class of regular cut-and-project sets with $m_\hatP>0$,
a corresponding lower bound on the gap length is present 
in the set of directions \eqref{WHXI} not only in the limit $T\to\infty$, but for \textit{any fixed} 
$T$:  
\begin{lem}\label{GAPMANIFESTLEM}
Let  $\scrP=\scrP(\scrW,\scrL)$ be a regular cut-and-project set in dimension $d=2$ such that
either $\bn\notin\scrP$ or $\bn\in\scrP^x$ for all $x\in X$, and furthermore
$\pi_\intl(\vecy)\notin\partial\scrW$ for all
$\vecy\in\scrL$ (viz., 
there are no ``singular vertices''; cf.\ \cite[p.\ 6]{BGHJ}).
Then for any non-proportional vectors $\vecp_1,\vecp_2\in\hatP$,
the triangle with vertices $\bn,\vecp_1,\vecp_2$
has area $\geq(\kappa_\scrP C_\scrP)^{-1} m_\hatP$.
In particular, for any $T>0$ and $1\leq j\leq\widehat N(T)$ we have
$\widehat\xi_{T,j}-\widehat\xi_{T,j-1}\geq\min(\frac12,(\pi\kappa_\scrP C_\scrP)^{-1}m_\hatP T^{-2})$.
\end{lem}
(Using the last bound of Lemma \ref{GAPMANIFESTLEM} together with 
$\widehat N(T)\sim \pi\kappa_\scrP C_\scrP T^2$ as $T\to\infty$
in the limit relation \eqref{GAPDISTRCORRES1} in Corollary \ref{GAPDISTRCOR},
we immediately recover the fact that $F(s)=1$ for $0\leq s\leq m_\hatP$.
We also remark that the condition $\bn\in\scrP^x$ for all $x\in X$ is fulfilled whenever $\bn\in\scrW$ and $\scrL$ is a lattice,
since then $H_g\subset\SL(n,\R)$.)
\begin{proof}
Assume that $\vecp_1,\vecp_2\in\hatP$ are non-proportial vectors
and that the triangle $\triangle(\bn,\vecp_1,\vecp_2)$
has area less than $(\kappa_\scrP C_\scrP)^{-1} m_\hatP$.
Note that for any $\vecp_1',\vecp_2'\in\R^2$ such that $\triangle(\bn,\vecp_1',\vecp_2')$ has the same area
and orientation as $\triangle(\bn,\vecp_1,\vecp_2)$, there exists $A\in\SL(2,\R)$ with
$\vecp_1'=\vecp_1A$ and $\vecp_2'=\vecp_2A$.
In particular there are some $A\in\SL(2,\R)$ and $\sigma_0\in(0,m_\hatP)$ such that
$\vecp_1A,\vecp_2A\in\fC(\kappa_\scrP^{-1}\sigma_0)$.
Now there are $\vecy_1,\vecy_2\in\scrL$ such that $\pi(\vecy_j)=\vecp_j$ and $\pi_\intl(\vecy_j)\in\scrW$
for $j=1,2$, and by assumption neither $\pi_\intl(\vecy_1)$ nor $\pi_\intl(\vecy_2)$ lie in $\partial\scrW$;
hence 
$\vecy_j\smatr{A}00{1_m}\in\fC(\kappa_\scrP^{-1}\sigma_0)\times\scrW^\circ$ for $j=1,2$.
It follows that $\#(\hatP^x\cap\fC(\kappa_\scrP^{-1}\sigma_0))\geq2$ for
$x=\Gamma\varphi_g(A)\in X$.
In fact, using our assumptions on $\scrP$ and 
the fact that $\fC(\kappa_\scrP^{-1}\sigma_0)\times\scrW^\circ$ is open,
we have $\#(\hatP^{x'}\cap\fC(\kappa_\scrP^{-1}\sigma_0))\geq2$ for all $x'$ in some open neighbourhood of 
$x=\Gamma\varphi_g(A)$
(cf.\ the proof of Lemma \ref{KEYCONTLEM2}).
However this violates our definition of $m_\hatP$.
We have thus proved the first part of the lemma.

To prove the second statement we merely have to note that 
$\widehat\xi_{T,j}-\widehat\xi_{T,j-1}=(2\pi)^{-1}\varphi(\vecp_1,\vecp_2)$ 
for some $\vecp_1\neq\vecp_2\in\hatP_T$.
If $\vecp_1,\vecp_2$ are not proportional then since $\triangle(\bn,\vecp_1,\vecp_2)$ has area
$\frac12\|\vecp_1\|\|\vecp_2\|\sin\varphi(\vecp_1,\vecp_2)<\frac12T^2\sin\varphi(\vecp_1,\vecp_2)$,
the first part of the lemma implies
$\varphi(\vecp_1,\vecp_2)>\sin\varphi(\vecp_1,\vecp_2)>2(\kappa_\scrP C_\scrP)^{-1}m_\hatP T^{-2}$;
on the other hand if $\vecp_1,\vecp_2$ are proportional then 
necessarily $\varphi(\vecp_1,\vecp_2)=\pi$.
\end{proof}

\begin{appendix}
\setcounter{section}{1}
\setcounter{equation}{0}
\section*{Appendix: Non-spherical truncations}\label{APPENDIX}

We have defined $\scrP_T$ as the intersection of $\scrP\setminus\{\bn\}$ and the ball $\scrB_T^d$.
However, as we will now explain, the fact that $\lambda$ is arbitrary in Theorem \ref{thm:main2} allows us to
obtain corresponding results with $\scrB_T^d$ replaced by a more general expanding domain.

Throughout this section, let $\scrP$ be an arbitrary point set with constant density $\theta(\scrP)$,
and let $\scrE$ be a starshaped region in $\R^d$ of the form
\begin{align*}
\scrE=\{r\vecv\col \vecv\in\S_1^{d-1},\:0\leq r<\ell(\vecv)\},
\end{align*}
where $\ell$ is a continuous function from $\S_1^{d-1}$ to $\R_{>0}$.
Set
\begin{equation}\label{NTEDEF}
\scrN_T(\sigma,\vecv,\scrP,\scrE):=\#\{\vecy\in\scrP\cap T\scrE\setminus\{\vecnull\} \col \|\vecy\|^{-1}\vecy\in \fD_T(\sigma,\vecv)\} .
\end{equation}
In particular, $\scrN_T(\sigma,\vecv,\scrP)=\scrN_T(\sigma,\vecv,\scrP,\scrB_1^d)$; cf.\ \eqref{disc00}.

It is natural to rescale $\sigma$ by a factor $\ell(\vecv)^{-d}$ in \eqref{NTEDEF},
since we then recover the property that the expectation value is asymptotically constant
and independent of the direction: 
For any probability measure $\lambda$ on $\S_1^{d-1}$ with continuous density, we have
\begin{equation}\label{expval2}
\lim_{T\to\infty}\int_{\S_1^{d-1}} \scrN_{T}(\ell(\vecv)^{-d}\sigma,\vecv,\scrP,\scrE)\, d\lambda(\vecv) =\sigma .
\end{equation}
This generalizes \eqref{expval}, and again follows from \eqref{asyden}. The proposition below covers both the rescaled and the original distribution.


\begin{prop}\label{GENEXPPROP}
Let $r\in\ZZ_{\geq 0}$. Assume that, for every $\sigma>0$ and every Borel probability measure $\lambda$  on $\S_1^{d-1}$ which is absolutely continuous with respect to $\omega$, the limit
\begin{equation}\label{GENEXPPROPRES1}
	\widetilde E(r,\sigma,\scrP):=\lim_{T\to\infty} \lambda(\{ \vecv\in\S_1^{d-1} : \scrN_{T}(\sigma,\vecv,\scrP,\scrB_1^d)\leq r \})
\end{equation}
exists and is continuous in $\sigma$ and independent of $\lambda$. Then, for every $\sigma$ and $\lambda$ as above, we have
\begin{equation}\label{GENEXPPROPRES2}
	\lim_{T\to\infty} \lambda(\{ \vecv\in\S_1^{d-1} : \scrN_{T}(\ell(\vecv)^{-d} \sigma  ,\vecv,\scrP,\scrE)\leq r \})=\widetilde E(r,\sigma,\scrP)
\end{equation}
and
\begin{equation}\label{GENEXPPROPRES3}
	\lim_{T\to\infty} \lambda(\{ \vecv\in\S_1^{d-1} : \scrN_{T}(\sigma ,\vecv,\scrP,\scrE)\leq r \})=\int_{\S_1^{d-1}} \widetilde E(r,\ell(\vecv)^d\sigma,\scrP) \; d\lambda(\vecv) .
\end{equation}
\end{prop}

In other words, the existence of a limit distribution of $\scrN_T(\sigma,\vecv,\scrP,\scrB_1^d)$ independent of $\lambda$ implies the existence the limit distributions of both
$\scrN_T(\sigma,\vecv,\scrP,\scrE)$ and $\scrN_T(\ell(\vecv)^{-d}\sigma,\vecv,\scrP,\scrE)$, where
the limit of the latter is in fact independent of $\scrE$! 

\begin{proof}
This is a relatively standard approximation argument.
We give the proof of \eqref{GENEXPPROPRES3}; the proof of \eqref{GENEXPPROPRES2} is very similar.

Let $r,\sigma,\lambda$ be given.
For $W$ an arbitrary measurable subset of $\S_1^{d-1}$,
let $\lambda_{|W}$ be the restriction of $\lambda$ to $W$,
and set $W_T:=\cup_{\vecv\in W}\fD_T(\sigma,\vecv)$
and $\ell_T^-=\inf\{\ell(\vecv)\col\vecv\in W_T\}$.
Then for any $0<T_0\leq T$ we have
\begin{align*}
\lambda_{|W} (\{\vecv\in\S_1^{d-1}\col\scrN_T(\sigma,\vecv,\scrP,\scrE)\leq r\})
&\leq\lambda_{|W}(\{\vecv\in\S_1^{d-1}\col\scrN_T(\sigma,\vecv,\scrP,\scrB_{\ell_T^-}^d)\leq r\})
\\
&=\lambda_{|W}(\{\vecv\in\S_1^{d-1}\col\scrN_{T\ell_T^-}(({\ell_T^-})^d\sigma,\vecv,\scrP,\scrB_1^d)\leq r\})
\\
&\leq\lambda_{|W}(\{\vecv\in\S_1^{d-1}\col\scrN_{T\ell_T^-}(({\ell_{T_0}^-})^d\sigma,\vecv,\scrP,\scrB_1^d)\leq r\}).
\end{align*}
If $\lambda(W)>0$ then we get, letting $T\to\infty$ and using \eqref{GENEXPPROPRES1} with $\lambda(W)^{-1}\lambda_{|W}$
in place of $\lambda$:
\begin{align*}
\limsup_{T\to\infty}\lambda_{|W}(\{\vecv\in\S_1^{d-1}\col\scrN_T(\sigma,\vecv,\scrP,\scrE)\leq r\})
\leq\lambda(W)\widetilde E(r,({\ell_{T_0}^-})^d\sigma,\scrP),
\end{align*}
for any $T_0>0$.
Letting here $T_0\to\infty$, we conclude
\begin{align}\label{GENEXPPROPSHORTPF1}
\limsup_{T\to\infty}\lambda_{|W}(\{\vecv\in\S_1^{d-1}\col\scrN_T(\sigma,\vecv,\scrP,\scrE)\leq r\})
\leq\lambda(W)\widetilde E(r,({\ell_W^-})^d\sigma,\scrP),
\end{align}
where $\ell_W^-=\inf_W\ell(\vecv)$.
Note that \eqref{GENEXPPROPSHORTPF1} also holds if $\lambda(W)=0$, trivially.

Given any $\ve>0$, since $\widetilde E(r,\sigma,\scrP)$ is continuous in $\sigma$ and $\ell(\vecv)$ is
uniformly continuous in $\vecv$,
we can find a partition of $\S_1^{d-1}$ into measurable subsets
$W_1,\ldots,W_m$ such that
$|\widetilde E(r,({\ell_{W_j}^-})^d\sigma,\scrP)-\widetilde E(r,\ell(\vecv)^d\sigma,\scrP)|<\ve$
for all $j\in\{1,\ldots,m\}$ and all $\vecv\in W_j$.
Using $\lambda=\sum_{j=1}^m\lambda(W_j)\lambda_{W_j}$,
and applying \eqref{GENEXPPROPSHORTPF1} for each $W_j$, we get
\begin{align*}
\limsup_{T\to\infty}\lambda(\{\vecv\in\S_1^{d-1}\col\scrN_T(\sigma,\vecv,\scrP,\scrE)\leq r\})
&\leq\sum_{j=1}^m\lambda(W_j)\widetilde E(r,({\ell_{W_j}^-})^d\sigma,\scrP)
\\
&<\int_{\S_1^{d-1}} \widetilde E(r,\ell(\vecv)^d\sigma,\scrP) \; d\lambda(\vecv)+\ve.
\end{align*}
Similarly one proves a corresponding lower bound for the $\liminf$.
Now \eqref{GENEXPPROPRES3} follows upon letting $\ve\to0$.
\end{proof}
\end{appendix}


\begin{thebibliography}{99}

\bibitem{Baake13}
M.\ Baake and U.\ Grimm, {\em Aperiodic order}. Vol. 1. A mathematical invitation. Encyclopedia of Mathematics and its Applications, 149. Cambridge University Press, Cambridge, 2013.
 
\bibitem{BGHJ}
M.\ Baake, F.\ G\"otze, C.\ Huck, and T.\ Jakobi,
Radial spacing distributions from planar point sets,
arXiv:1402.2818.

\bibitem{Boca00}
F.P. Boca, C. Cobeli and A. Zaharescu, Distribution of lattice points visible from the origin.  Comm. Math. Phys.  \textbf{213}  (2000), 433--470.

\bibitem{Borevich}
A.I.\ Borevich and I.R.\ Shafarevich,  {\em Number theory}. Pure and Applied Mathematics, Vol. 20, Academic Press, New York-London 1966.

\bibitem{bruijn}
N.\ G.\ de Bruijn, Algebraic theory of Penrose's non-periodic tilings of the plane,
\textit{Koninklijke Nederlandse Akademie van Wetenschappen. Indagationes Mathematicae}
\textbf{43} (1981), 39--52,\: 53--66.
 
\bibitem{Hof98}
A.\ Hof, Uniform distribution and the projection method,
in \textit{Quasicrystals and discrete geometry ({T}oronto, {ON}, 1995)},
Fields Inst. Monogr. \textbf{10}, (1998), 201--206.

\bibitem{lang}
S.\ Lang, \textit{Algebraic Number Theory}, Springer-Verlag, New York, 1994.

\bibitem{Marklof07}
J. Marklof, Distribution modulo one and Ratner's theorem, \textit{Equidistribution in Number Theory, An Introduction,} eds. A. Granville and Z. Rudnick, Springer 2007, pp. 217-244.

\bibitem{partI}
J. Marklof and A. Str\"ombergsson, The distribution of free path lengths in
the periodic Lorentz gas and related lattice point problems,
Annals of Math. \textbf{172} (2010), 1949--2033.

\bibitem{qc}
J. Marklof and A. Str\"ombergsson, 
Free path lengths in quasicrystals, Comm. Math. Phys.  {\bf 330} (2014),  723--755.


\bibitem{meyer}
Y.\ Meyer, \textit{Algebraic numbers and harmonic analysis}, North-Holland Publishing Co., Amsterdam, 1972.

\bibitem{meyer95}
Y.\ Meyer, Quasicrystals, Diophantine approximation and algebraic numbers,
in \textit{Beyond quasicrystals (Les Houches, 1994)}, pp.\ 3--16, Springer, Berlin, 1995.

\bibitem{Moody97}    
R.V.\ Moody, Meyer sets and their duals. The mathematics of long-range aperiodic order (Waterloo, ON, 1995), 403--441, NATO Adv. Sci. Inst. Ser. C Math. Phys. Sci., 489, Kluwer Acad. Publ., Dordrecht, 1997. 

\bibitem{Moody00}  
R.V.\ Moody, Model set: A survey,  
From Quasicrystals to More Complex Systems
Centre de Physique des Houches Volume 13, 2000, pp. 145--166
 
\bibitem{Moody02}    
R.V.\ Moody,  Uniform distribution in model sets, Canad. Math. Bull. {\bf 45} (2002), 123--130.     
    
\bibitem{Pleasants03}
P.A.B.\ Pleasants,
Lines and planes in 2- and 3-dimensional quasicrystals, in \textit{Coverings of discrete quasiperiodic sets},
Springer Tracts Modern Phys., \textbf{180}, pp.\ 185--225,
Springer, Berlin, 2003.

\bibitem{Ratner91a}
M.\ Ratner, On Raghunathan's measure conjecture, Ann. of Math. {\bf 134} (1991), 545--607.

\bibitem{Ratner91b}
M.\ Ratner, Raghunathan's topological conjecture and distributions of unipotent flows,
Duke Math.\ J.\ \textbf{63} (1991), 235--280.

\bibitem{Raghunathan}
M.\ S.\ Raghunathan, \textit{Discrete subgroups of Lie groups}, Springer-Verlag, New York, 1972.

\bibitem{mS95}
M.\ Senechal, \textit{Quasicrystals and geometry}, Cambridge University Press, Cambridge, 1995.

%

\bibitem{Schlottmann}
M.\ Schlottmann, Cut-and-project sets in locally compact abelian groups,
in \textit{Quasicrystals and discrete geometry ({T}oronto, {ON}, 1995)},
Fields Inst. Monogr. \textbf{10}, (1998), 247--264.

\bibitem{veech}
W. A. Veech, Siegel measures, Ann.\ of Math.\ \textbf{148} (1998), 895--944.

\bibitem{Weil}
A. Weil, \textit{Basic Number Theory,} 3rd ed., Springer-Verlag, New York, 1974.

\bibitem{wellner}
Aad W.\ van der Vaart and Jon A.\ Wellner, \textit{Weak convergence and empirical processes},
Springer-Verlag, New York, 1996.
\end{thebibliography}
\end{document}